\documentclass[11pt]{amsart}   	 
\usepackage{amscd, amssymb, mathrsfs, mathabx, tikz-cd, amsthm, thmtools, stmaryrd}
\usepackage{hyperref}

\input{xy}
\xyoption{all}

\numberwithin{equation}{section}

\newcommand{\CC}{\mathbb{C}}
\newcommand{\EE}{\mathbb{E}}

\newcommand{\PP}{\mathbb{P}}
\newcommand{\QQ}{\mathbb{Q}}
\newcommand{\RR}{\mathbb{R}}
\newcommand{\ZZ}{\mathbb{Z}}

\newcommand{\GG}{\mathbb{G}}

\newcommand{\cal}{\mathcal}

\def\cC{{\cal C}}

\def\cH{{\cal H}}

\def\cL{{\cal L}}
\def\cM{{\cal M}}

\def\cO{{\cal O}}

\def\cV{{\cal V}}

\def\fB{\mathfrak{B}}
\def\fC{\mathfrak{C}}

\def\fM{\mathfrak{M}}

\def\fS{\mathfrak{S}}
\def\loc{\mathrm{loc}}

\def\and{\quad{\rm and}\quad}
\def\lra{\longrightarrow }

\def\hra{\hookrightarrow}

  \DeclareMathOperator{\Hom}{Hom}
 
\DeclareMathOperator{\image}{Im} 
\DeclareMathOperator{\id}{id}

\DeclareMathOperator{\rank}{rank}
\DeclareMathOperator{\spec}{Spec}

\newtheorem{prop}{Proposition}[section]
\newtheorem{theo}[prop]{Theorem}
\newtheorem{lemm}[prop]{Lemma}
\newtheorem{coro}[prop]{Corollary}

\newtheorem{defi}[prop]{Definition}

\newtheorem{const}[prop]{Construction}

\def\beq{\begin{equation}}
\def\eeq{\end{equation}}

\def\dual{^{\vee}}

\def\virt{^{\mathrm{vir}} }
\def\virtloc{\virt_\loc}

\def\hra{\hookrightarrow}

\def\bar{\overline}
\def\rou{\partial}

\def\wtil{\widetilde}
\def\what{\widehat}

\def\cseq{^{\bullet}}
\def\hseq{_{\bullet}}

\def\ker{\mathrm{ker}}

\def\deg{\mathrm{deg}}
\def\dim{\mathrm{dim}}

\def\c>{\succ}
\def\c<{\prec}

\def\l({\left(}
\def\r){\right)}

\newcommand{\ses}[3]{0  \to #1 \to #2 \to #3 \to 0}
\newcommand{\Gysin}[1]{0_{#1}^!}

\newcommand{\bdst}[1]{h^1/h^0(#1)}

\title[Algebraic reduced genus one GW invariants]{Algebraic reduced genus one Gromov-Witten invariants for complete intersections in projective spaces}
\author{Sanghyeon Lee}
\address{KIAS, Seoul, Korea}
\email{sanghyeon@kias.re.kr}
\author{Jeongseok Oh}
\address{KIAS, Seoul, Korea}
\email{jeongseok@kias.re.kr}

\thanks{}

\date{}

\begin{document}

\begin{abstract}
In \cite{Zin08standard, Zin09red}, Zinger defined reduced Gromov-Witten (GW) invariants %in symplectic geometry 
and proved a comparison theorem of standard and reduced genus one GW invariants for every symplectic manifold (with all dimension). 
In \cite{CL15}, Chang and Li  
provided a proof of the comparison theorem for quintic Calabi-Yau 3-folds in algebraic geometry by taking a definition of reduced invariants as an Euler number of certain vector bundle.
In \cite{CM18}, Coates and Manolache have defined reduced GW invariants in algebraic geometry following the idea by Vakil and Zinger in \cite{VZ07natural} and proved the comparison theorem for every Calabi-Yau threefold. 
In this paper, we prove the comparison theorem for every (not necessarily Calabi-Yau) complete intersection of dimension 2 or 3 in projective spaces by taking a definition of reduced GW invariants in \cite{CM18}. 
%When cohomology classes are from projective spaces, we prove that reduced GW invariants can be obtained by an Euler class of some suitable vector bundle.
\end{abstract}

\maketitle
\setcounter{tocdepth}{1}
\tableofcontents

\section{Introduction}

Throughout the paper, we will work over the base field $\CC$.
Let $Q$ be a smooth projective variety.
The moduli space of stable maps $\bar{M}_{g,k}(Q,d)$ carries a canonical virtual fundamental class $[\bar{M}_{g,k}(Q,d)]\virt$ of virtual dimension $c_1 (T_Q) \cap d + (1-g)(\dim Q-3)+k$. For a cohomology class $\alpha$ of $Q^k := Q \times \cdots \times Q$, $Q^0:=\spec \CC$, one can define a GW invariant
\begin{align*}
GW_{g,d}(\alpha) := \int_{[\bar{M}_{g,k}(Q,d)]\virt} ev^*\alpha \in \QQ
\end{align*} 
where $ev: \bar{M}_{g,k}(Q,d) \rightarrow Q^k$ is the evaluation morphism. We will denote $GW_{g,d}(1_{\spec \CC})$ by $GW_{g,d}$ for $k=0$, $1_{\spec \CC} \in H^0(\spec \CC)$.
A reduced sublocus $\bar{M}^{red}_{1,k}(Q,d) \subset \bar{M}_{1,k}(Q,d)$ for $g=1$ is defined to be $\bar{M}_{1,k}(Q,d) \cap \bar{M}^{red}_{1,k}(\PP^n,d)$ where $\bar{M}^{red}_{1,k}(\PP^n,d)$ is the closure of $M_{1,k}(\PP^n,d) \subset \bar{M}_{1,k}(\PP^n,d)$.
Here, $M_{1,k}(\PP^n,d)$ is a moduli space of stable maps to $\PP^n$ from a {\em smooth} domain curves.
The reduced GW invariants $GW_{1,d}^{red}$ and $GW_{1,d}^{red}(\alpha)$ are defined to be an integration
\begin{align*}
GW^{red}_{1,d}(\alpha) := \int_{[\bar{M}^{red}_{1,k}(Q,d)]\virt} ev^*\alpha \in \QQ
\end{align*}
for some suitable choice of a class $[\bar{M}^{red}_{1,k}(Q,d)]\virt$ with the virtual dimension as the one of $[\bar{M}_{g,k}(Q,d)]\virt$.
In symplectic geometry, the class $[\bar{M}^{red}_{1,k}(Q,d)]\virt$ is constructed to be a fundamental class of a deformed space of $\bar{M}^{red}_{1,k}(Q,d)$ by Zinger in \cite{Zin09red}. 
In \cite{CM18}, Coates and Manolache have constructed a virtual class 
$[\bar{M}^{red}_{1,k}(Q,d)]\virt \in A_*(\bar{M}^{red}_{1,k}(Q,d))_\QQ$ 
following the idea of Vakil and Zinger in \cite{VZ07natural}. 
In this article, we will use the definition of $[\bar{M}^{red}_{1,k}(Q,d)]\virt$ in \cite{CM18}.
We will introduce a brief explanation of this definition in Section \S \ref{desingAndDecomp}.

Our main result states the following:

\begin{theo} \label{main}
For a complete intersection $Q \subset \PP^n$ in projective space and $\alpha \in H^*(Q)^k$, we have
\begin{align*}
GW_{1,d}(\alpha)-GW^{red}_{1,d}(\alpha)= \left\{
\begin{array}{cc}
0 & \dim Q=2, \\
\frac{2-c_1(T_Q) \cap d[line]}{24}GW_{0,d}(\alpha) & \dim Q=3,
\end{array} \right.
\end{align*}
where $[line]$ is the Poincar\'e dual of the hyperplane class in $Q$.
\end{theo}

\noindent In symplectic side, there is a more general formula which compares $GW_{1,d}(\alpha)$ and $GW_{1,d}^{red}(\alpha)$ for all symplectic manifold $Q$ with all dimensions proved by Zinger in \cite{Zin08standard}.

On the other hand, if $Q\subset \PP^n$ is a complete intersection in $\PP^n$, by \cite{LZ09, Zin07structure, VZ08, Pop13}, reduced invariants (in the sense of \cite{Zin09red}) $GW^{red}_{1,d}(\alpha)$ can be obtained by Euler classes of some suitable vector bundles when $\alpha$ is a pullback class from $(\PP^n)^k$.
This property is called a {\em hyperplane property} or {\em quantum Lefschetz property}.
We note that the description of Euler number makes sense in algebraic geometry.
Using this description, in \cite{CL15}, Chang and Li proved the following in algebraic geometry:
for a quintic threefold $Q$ and for $d>0$, 
\begin{align}
\label{CL15}
GW_{1,d}-GW^{red}_{1,d}=\frac{1}{12}GW_{0,d}.
\end{align}
%The connection to \eqref{CL15} is the following:

%\begin{theo} \label{QLP}
%The reduced invariants which we defined here satisfy hyperplane property. 
%More precisely, the class $[\bar{M}^{red}_{1,k}(Q,d)]\virt$ can be written as a (refined) Euler class of some vector bundle on $\wtil{\cM}^{red}_{k,d}$ which will defined in (4) Section \ref{desingAndDecomp}.
%\end{theo}

%\noindent We note that in Theorem \ref{QLP}, we allow cohomology classes $\alpha$ defining reduced invariants to be classes of $Q^k$.

\noindent In \cite{CM18}, Coates and Manolache proved this result for every smooth projective Calabi-Yau 3-fold $Q$; using their definition of reduced GW invariants.

\subsection{Desingularization and Decomposition} \label{desingAndDecomp}

From now on, we assume that $Q = \{ f_1 = \cdots = f_m = 0\} \subset \PP^n$ is a complete intersection in $\PP^n$ with codimension $m$ defined by equations $f_1$, ..., $f_m$.
For $g=1$, there is a space $\wtil{\cM}_{k,d}$ and a morphism $b: \wtil{\cM}_{k,d} \to \bar{M}_{1,k}(\PP^n,d)$ %over $\bar{M}_{1,k}(\PP^n,d)$ 
such that
\begin{enumerate} 
\item $b$ is a proper,
\item the complex defined by a pullback of the natural perfect obstruction theory on $\bar{M}_{1,k}(\PP^n,d)$,
$$
(\RR \pi_{*}ev^*\cO_{\PP^n}(1)^{\oplus (n+1)})^\vee $$ 
is again a perfect obstruction theory on $\wtil{\cM}_{k,d}$ where $\pi$ is a projection from a universal curve and $ev$ is a universal morphism,
\item $\wtil{\cM}_{k,d}$ is {\em locally} defined by an equation $F: R^0 \subset R \times \CC^n \to \CC^n$ where $R$ is a smooth space; $R^0 \subset R \times \CC^n$ is an open subset; and the function $F$ has a form of 
$$F(r, w_1, ..., w_n)= (g(r)w_1, ..., g(r)w_n)$$
for some function $g$ on $R$,
\item the spaces $\{ w_1 = \cdots = w_n =0 \}$ (resp. the spaces $\{g(r)=0\}$) glue to a subspace $\wtil{\cM}^{red}_{k,d} \subset \wtil{\cM}_{k,d}$ (resp. a subspace $\wtil{\cM}^{rat}_{k,d} \subset \wtil{\cM}_{k,d}$) and there is a decomposition 
$$\wtil{\cM}_{k,d} = \wtil{\cM}^{red}_{k,d} \cup \wtil{\cM}^{rat}_{k,d}.$$
Moreover, $\wtil{\cM}^{red}_{k,d}$ is irreducible, smooth and its image by $b$ is $\bar{M}_{1,k}^{red}(\PP^n,d)$. 
\end{enumerate}

The morphism $b$ is the same as a desingularization constructed by Vakil and Zinger in \cite{VZ08}.
In \cite{HL10}, Hu and Li found local defining equation for $k=0$. 
We note that in \cite[Remark 5.7]{HL10}, Hu and Li mentioned that their results can be extended when the marked points exist. 
For a desingularization $b$ defined by Vakil and Zinger, the condition (1) above is trivial and (2) is satisfied by \cite[Proposition 7.2]{BF97}.
In Section 2, we will provide a detail about conditions (3) and (4) mentioned in the last paragraph of \cite{HL10}. 

In \cite{CM18}, the class $[\bar{M}^{red}_{1,k}(Q,d)]\virt$ is defined as follows:
Let $\pi: \cC \to \wtil{\cM}_{k,d}^{red}$ be the universal curve and $ev: \cC \to \PP^n$ be the evaluation morphism.
It is known that for $l>0$, $\pi_* ev^* \cO_{\PP^n}(l)$ is locally free by \cite[Theorem 2.11]{HL10}.
Let $N^{red} := \pi_* ev^* ( \oplus_{i=1}^m \cO_{\PP^n}(\deg f_i) ) |_{\wtil{\cM}^{red}_{Q, k,d}}$ be a vector bundle on a space $\wtil{\cM}^{red}_{Q,k,d} := \wtil{\cM}^{red}_{k,d} \times_{\bar{M}_{1,k}(\PP^n,d)^{red}}\bar{M}_{1,k}(Q,d)^{red}$.
It is proven that the normal cone $C_{\wtil{\cM}^{red}_{Q,k,d} / \wtil{\cM}^{red}_{k,d} }$ is embedded in a total space of $N^{red}$ in \cite{CM18}.

\begin{defi}[\cite{CM18}] \label{Reduced::Class}
The virtual fundamental class $[\bar{M}^{red}_{1,k}(Q,d)]\virt$ is defined by 
\begin{align*}
[\bar{M}^{red}_{1,k}(Q,d)]\virt := b'_*(0^!_{N^{red}}[C_{\wtil{\cM}^{red}_{Q,k,d} / \wtil{\cM}^{red}_{k,d} }]) \in A_*(\bar{M}^{red}_{1,k}(Q,d))_\QQ
\end{align*}
where $b' : \wtil{\cM}^{red}_{Q,k,d} \to \bar{M}_{1,k}^{red}(Q,d)$ is a restriction of $b$.
\end{defi}

\noindent By definition, when $\alpha$ is a pullback class from $(\PP^n)^k$ we note that $GW^{red}_{1,d}(\alpha)$ is equal to the degree of a class $e(N^{red})\cap[\wtil{\cM}^{red}_{k,d}]\cap ev^*\alpha$ in $A_0(\wtil{\cM}^{red}_{k,d})_\QQ$. 
Here, $e(-)$ stands for an Euler class. 
This expression is sometimes helpful; see \cite{LZ09, Zin09CY, Pop13}.

\subsection{Moduli space with fields}
The moduli spaces of stable maps with fields $\bar{M}_{g,k}(\PP^n,d)^{p}$ have constructed by Chang and Li in \cite{CL11fields}. 
Recall that $\bar{M}_{g,k}(\PP^n,d)$ is a moduli space consisting of $(C, \cL, u)$ where $C$ is a prestable genus $g$ curve with $k$ marked points; $\cL$ is a degree $d$ line bundle on $C$ and $u=(u_0,\dots,u_n) \in H^0(C, \cL^{\oplus (n+1)})$ which satisfies a certain stability condition. 
By definition, $\bar{M}_{g,k}(\PP^n,d)^{p}$ is a moduli space consisting of $(C, \cL, u,p)$ where $(C, \cL, u)$ is an object in $\bar{M}_{g,k}(\PP^n,d)$ and $p=(p_1,\dots,p_m) \in H^0(C, (\oplus_i \cL^{\otimes -\deg f_i}) \otimes \omega_C)$. 
Here, $\omega_C$ is a dualizing sheaf of $C$.
We note that $\bar{M}_{g,k}(\PP^n,d)^{p}$ does depend on degrees of $f_i$, not $Q$.

By using cosection localization \cite{KL13cosec}, one can define a class $[\bar{M}_{g,k}(\PP^n,d)^{p}]\virtloc$ in Chow group of $\bar{M}_{g,k}(Q,d)$ which has the same virtual dimension as the one of $[\bar{M}_{g,k}(Q,d)]\virt$. 
For quintic threefold $Q$ and for $k=0$, both virtual dimensions are zero. In \cite{CL11fields}, Chang and Li proved that 
\begin{align*}
\deg[\bar{M}_{g,0}(\PP^n,d)^{p}]\virtloc=(-1)^{5d+1-g}\deg[\bar{M}_{g,0}(Q,d)]\virt \in \QQ,
\end{align*} for $d>0$ and for quintic threefold $Q$. 
For the proof of (\ref{CL15}) in \cite{CL15}, Chang and Li used this result. 
In \cite{KO18localized, CL18inv}, it is proven that
\begin{align}
\label{KO18}
[\bar{M}_{g,k}(\PP^n,d)^{p}]\virtloc & = (-1)^{d\sum_i \deg f_i +m-mg}  [\bar{M}_{g,k}(Q,d)]\virt \\
& \in A_{c_1(T_Q) \cap d + (1-g)(\dim Q-3)+k}(\bar{M}_{g,k}(Q,d))_\QQ, \nonumber
\end{align}
for any complete intersection $Q =\{f_1=\cdots =f_m=0 \} \subset \PP^n$.
By \eqref{KO18}, Theorem \ref{main} can be restated with invariants using the cycle $[\bar{M}_{g,k}(\PP^n,d)^{p}]\virtloc$ which we believe more easier to handle than the cycle $[\bar{M}_{g,k}(Q,d)]\virt$.

By a desingularization in Section \S \ref{desingAndDecomp}, for $g=1$, we have 
$$b : \wtil{\cM}^p_{k,d}:= \wtil{\cM}_{k,d} \times_{\bar{M}_{1,k}(\PP^n,d)} \bar{M}_{1,k}(\PP^n,d)^p \to \bar{M}_{1,k}(\PP^n,d)^p$$ 
which satisfies
\begin{enumerate}
\item the complex defined by a pullback of the natural perfect obstruction theory on $\bar{M}_{1,k}(\PP^n,d)^p$
$$
(\RR \pi_{*}ev^*\cO_{\PP^n}(1)^{\oplus (n+1)} \bigoplus \oplus_i \RR \pi_{*}(ev^*\cO_{\PP^n}(-l_i) \otimes \omega_{\pi} ))^\vee
$$
is again a perfect obstruction theory on $\wtil{\cM}^p_{k,d}$ where $\pi$ is a projection from a universal curve and $ev$ is a universal morphism, and the localized virtual cycle defined by a pull-back of a cosection (introduced in Section \S 3 more precisely) defines a class $[\wtil{\cM}^p_{k,d}]\virtloc $.
Moreover, it satisfies an equivalence of classes
\begin{align} \label{VirPushFor}
b_*[\wtil{\cM}^p_{k,d}]\virtloc = [\bar{M}_{1,k}(\PP^n,d)^{p}]\virtloc,
\end{align}
\item $\wtil{\cM}_{k,d}^p$ is {\em locally} defined by an equation $\wtil{F}: \wtil{R} \subset R \times \CC^n \times \CC^m \to \CC^{n+m}$ where $R$ is the same space as in (3) of Section \S \ref{desingAndDecomp}; $\wtil{R} \subset R \times \CC^{n+m}$ is an open subset; and the function $\wtil{F}$ has a form of 
$$\ \ \ \ \ \ \  \ \wtil{F}(r, w_1, ..., w_n, t_1, ..., t_m)= (g(r)w_1, ..., g(r)w_n, g(r)t_1, ..., g(r)t_m).$$
Here, the function $g$ on $R$ is defined in (3) Section \S \ref{desingAndDecomp}.
\end{enumerate}

Let $\fC^p_{k,d}$ be an intrinsic normal cone of $\wtil{\cM}^p_{k,d}$ and $\fC^{p, red}_{k,d}$ be a (unique) irreducible component supported on the gluing of $\{w_1 = ... = w_n = 0\}$. 
Let $\fC^{p, rat}_{k,d}$ be the union of rest components in $\fC^p_{k,d}$.
A decomposition of a cycle $[\fC^p_{k,d}]=[\fC^{p, red}_{k,d}]+[\fC^{p, rat}_{k,d}]$ gives us a decomposition of a class $[\wtil{\cM}^p_{k,d}]\virtloc = A^{red} + A^{rat}$.
The precise descriptions of $A^{red}$ and $A^{rat}$ are given in the end of Section \S 3.
We will show that 
\begin{align}
(-1)^{d\sum_i \deg f_i} b_*A^{red}  &=  [\bar{M}^{red}_{1,k}(Q,d)]\virt   \label{A:red}\\
(-1)^{d\sum_i \deg f_i}  b_*A^{rat} \cap ev^*\alpha &= \left\{
\begin{array}{cc}
0 & \dim Q=2, \\
\frac{2-c_1(T_Q) \cap d[line]}{24}GW_{0,d}(\alpha) & \dim Q=3.
\end{array} \right. \label{A:rat}
\end{align}
We obtain Theorem \ref{main} by combining \eqref{KO18}, \eqref{VirPushFor}, \eqref{A:red} and \eqref{A:rat}.
We will discuss the conditions (1) and (2) for $b: \wtil{\cM}^p_{k,d} \to \bar{M}_{1,k}(\PP^n,d)^p$
in Section 3.
We will show \eqref{A:red} (resp. \eqref{A:rat}) in Section 4 (resp. Section 5).

\subsection{Remark}

When $k$ is not zero, the proof has different feature from the case when $k$ is zero.
A degree of a cycle is preserved by a proper pushforward.
Hence, for a calculation, it is enough to know its pushforward expression into a better space.
We don't need to really know a precise cycle.
But, to define a cap product, we really need a cycle.

\medskip

\noindent {\bf Acknowledgments}
The authors would like to thank Aleksey Zinger for kind comments about a history of the related works.
We would like to thank Cristina Manolache for valuable conversations. We also thank Mu-Lin Li for pointing out a mistake in the proof of Proposition \ref{conedecomposition:RED} in the earlier version of this paper.

\section*{Notations}

\noindent The following is a table of notations which will be frequently used.
\begin{small}
\begin{table}[ht]
\label{eqtable}
\renewcommand\arraystretch{1.0}
\noindent\[
\begin{array}{l|l}
\hline
Q & \text{a complete intersection in } \PP^n \text{ defined by } \{f_1=\dots=f_m=0\}  \\
\pi: \cC \to X & \text{a universal curve on a stack $X$} \\
ev & \text{an evaluation morphism from $\cC$} \\
\fM_{g,k} & \text{the moduli space of prestable genus $g$ curves with $k$ marked}\\
& \text{points} \\
\wtil{\fM}_{1,k} \text{ or } \wtil{\fM} & \text{the Vakil-Zinger blow-up of $\fM_{1,k}$}\\
\fM_{g,k}^w & \text{the moduli space parametrizing $(C,w)$ where $C \in \fM_{g,k}$ and} \\
& \text{$w$ is a weight function on irreducible components of $C$}\\
\end{array}
\]
\end{table}
\end{small}

\begin{small}
\begin{table}[ht]
%\label{eqtable}
\renewcommand\arraystretch{1.0}
\noindent\[
\begin{array}{l|l}
\fB un^{g,k}_{\CC^*} \text{ or } \fB & \text{the moduli space parametrizing $(C,L)$ where $C \in \fM_{g,k}$ and} \\
& \text{$L$ is a line bundle on $C$}\\
\wtil{\fB}_{1,k} \text{ or } \wtil{\fB} & \wtil{\fM} \times_{\fM_{1,k}} \fB \\
T_A & \text{a tangent sheaf of $A$} \\
T_{A/B} & \text{a tangent sheaf of $A$ relative to $B$} \\
C_{A/B} & \text{a normal cone to $A$ in $B$} \\
N_{A/B} & \text{a normal sheaf to $A$ in $B$} \\
\wtil{\cM}_{k,d} \text{ or } \wtil{\cM} & \wtil{\fM} \times_{\fM_{1,k}} \bar{M}_{1,k}(\PP^n,d)\\
\wtil{\cM}^{red}_{k,d} \text{ or } \wtil{\cM}^{red} & \text{the reduced component of $\wtil{\cM}$}\\
\wtil{\cM}^{rat}_{k,d} \text{ or } \wtil{\cM}^{rat} & \text{the rational component of $\wtil{\cM}$}\\
\wtil{\cM}^{\mu}  &\text{an irreducible component of $\wtil{\cM}^{rat}$ indexed by $\mu$} \\
\wtil{\cM}_{Q,k,d} \text{ or } \wtil{\cM}_Q & \wtil{\fM} \times_{\fM_{1,k}} \bar{M}_{1,k}(Q,d)\\
\wtil{\cM}^{red}_{Q,k,d} \text{ or } \wtil{\cM}_Q^{red} & \wtil{\cM}^{red} \times_{\bar{M}_{1,k}(\PP^n,d)} \bar{M}_{1,k}(Q,d) \\
\wtil{\cM}^{rat}_{Q,k,d} \text{ or } \wtil{\cM}_Q^{rat} & \wtil{\cM}^{rat} \times_{\bar{M}_{1,k}(\PP^n,d)} \bar{M}_{1,k}(Q,d)\\
\bar{M}_{g,k}(\PP^n,d)^p & \text{the moduli space of genus $g$, degree $d$ stable map to $\PP^n$} \\
& \text{with $k$ marked points and $p$-fields} \\ 
\wtil{\cM}^p_{k,d} \text{ or } \wtil{\cM}^p & \wtil{\cM}_{k,d} \times_{\bar{M}_{1,k}(\PP^n,d)} \bar{M}_{1,k}(\PP^n,d)^p  \\
\wtil{\cM}^{p,red}  &\text{the reduced component of $\wtil{\cM}^p$} \\
\wtil{\cM}^{p}_{red}  &\wtil{\cM}^{red} \times_{\wtil{\cM}} \wtil{\cM}^p \\
\wtil{\cM}^{p,rat}  &\wtil{\cM}^{rat} \times_{\wtil{\cM}} \wtil{\cM}^p \\
\wtil{\cM}^{p,\mu}  &\text{an irreducible component of $\wtil{\cM}^{p,rat}$ indexed by $\mu$} \\
\bdst{E_0 \to E_1} & \text{a bundle stack $[E_1/E_0]$}\\
\Gysin{\bdst{E},\sigma} \text{ or } \Gysin{E,\sigma} & \text{a localized Gysin map}\\
|E| & \text{the total space of a vector bundle $E$}\\
Z(\alpha) & \text{a degeneracy locus of a cosection $\alpha$, or equivalently $(\alpha^\vee)^{-1}(0)$}\\
E(\alpha) & \text{zero of a function $w_\alpha: |E| \to \CC$ induced by}\\ 
& \text{a cosection $\alpha:E \to \cO$}\\
\fC^p & \text{the intrinsic normal cone of $\wtil{\cM}^p$ relative to $\wtil{\fB}$} \\
\fC^{p,red} & \text{an irreducible component of $\fC^p$ supported on $\wtil{\cM}^{p,red}$} \\
\fC^{p,\mu} & \text{an irreducible component of $\fC^p$ supported on $\wtil{\cM}^{p,\mu}$} \\
\fC^{p,\mu}_\Delta & \text{a union of components of $\fC^p$ supported on $\wtil{\cM}^{p}_{red} \cap \wtil{\cM}^{p,\mu}$} \\
\hline
\end{array}
\]
\end{table}
\end{small}

\medskip

\section{Hu-Li's local equations of Vakil-Zinger's desingularization; review and generalization to cases with marked points.}\label{hulilocaleq}

To explain local defining equations of a desingularization, we introduce some terminologies in combinatorics first.

A {\em rooted tree} is a pair $(\Gamma, \star)$ where $\Gamma$ is a tree which has a partial ordering given by descendant relation on the set of vertices $Ver(\Gamma)$ and $\star \in Ver(\Gamma)$ is a vertex, called a {\em root}, such that every vertex except for $\star$ is a descendant of $\star$. 
In other words, every vertex except for $\star$ has exactly one direct ascendant and $\star$ is an ascendant of every other vertices.
A tree $\Gamma$ is allowed to have half-edges, called {\em legs}.
For a rooted tree $(\Gamma, \star)$, a vertex is called $terminal$ if it has no descendants. We denote the set of terminal vertices by $Ver(\Gamma)^t$, and the set of non-root vertices by $Ver(\Gamma)^* =Ver(\Gamma) \backslash \{\star \}$.
A {\em weighted rooted tree} is a pair $(\Gamma, \star, w)$ (or, simply, $(\Gamma, w)$) where $(\Gamma, \star)$ is a rooted tree and $w$ is a function $w: Ver(\Gamma) \to \ZZ_{\geq 0}$. %We call a vertex $v$ of $\Gamma$ is ghost if $w(v)=0$.
A weighted rooted tree $(\Gamma, w)$ is called {\em terminally weighted} if a vertex has a positive weight if and only if it is terminal.

Let $(\Gamma, \star)$ be a terminally weighted rooted tree. 
A {\em trunk} of $\Gamma$ is the maximal path $\star=v_0 \c< v_1 \c< \dots \c< v_r$ such that $v_{i+1}$ is the only direct descendant of $v_i$ for $0 \leq i \leq r-1$. 
Here, the convention is that $x < y$ if and only if $x$ is an ascendant of $y$.
Note that $r$ can be zero. 
We call $v_r$ the {\em branch vertex} of $\Gamma$.
We abbreviate the trunk of $\Gamma$ by $\bar{\star v_r}$. When $v_r=\star$ we say that $\Gamma$ has no trunk.

Let $\bar{\star v_r}$ be the trunk of $\Gamma$. 
If $\Gamma$ is not a path tree, i.e., $\bar{\star v_r} \subsetneq \Gamma$, then $\Gamma$ is obtained by attaching $l>1$ rooted trees $\gamma_1,\dots,\gamma_l$ to $v_r$ which are called {\em branches}. 
The following picture is an example of a rooted tree
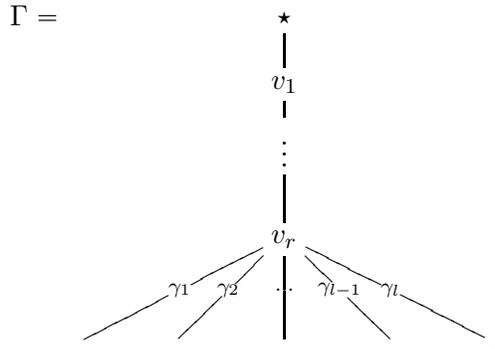
\[\Gamma=\xymatrix@=1pc{& & & & \star \ar@{-}[d] & & & & \\
& & & & v_1 \ar@{-}[d] & & & & \\
& & & & \vdots \ar@{-}[d] & & & & \\
& & & & v_r \ar@{-}[ddllll]|{\gamma_1} \ar@{-}[ddll]|{\gamma_2} \ar@{-}[dd]|{\dots} \ar@{-}[ddrr]|{\gamma_{l-1}} \ar@{-}[ddrrrr]|{\gamma_l} & & & & \\
 &  &  &  &  &  &  &  &  \\
 &  &  &  &  &  &  &  &.
}
\]

%We denote the number of branches by $br(\Gamma)$.
%We define $br(\Gamma)$ to be zero if $\Gamma$ is a path tree. 
%We note that $br(\Gamma)$ cannot be 1.
%We define an {\em index} $ind(\Gamma)$ of terminally weighted rooted tree $\Gamma$ by the sum of $br(\Gamma)$ and numbers of legs attached to the vertices in the trunk. 
%When we denote the number of legs attached to the vertex $v$ by $leg(v)$, we have
%\begin{align*}
%ind(\Gamma)  =br(\Gamma)+leg(v_0)+\dots+leg(v_r) 
%\end{align*}
%where $v_0=\star, ..., v_r$ are vertices on the trunk of $\Gamma$.

For a weighted genus one nodal curve with $k$-marked points $(C,w)\in \fM^w_{1,k}$, we associated a terminally weighted rooted tree $\Gamma_C$ with $k$ legs. 
First, we consider a dual graph associated with a curve. 
Marked points correspond to legs. 
The minimal genus one connected component may correspond to a cycle in a dual curve, then contract it to the unique vertex which will be the root.
%Then, the graph becomes a tree. 
%After this, we obtain a tree and there is a unique vertex in the tree corresponding to a minimal genus one component. 
%We define it to be the root of this tree. 
Finally, to make this tree to be terminally weighted, we do some pruning process; see \cite[Section 3.4]{HL10} for the pruning process.

For later use, we introduce an {\em advancing} operation at a certain vertex by presenting a specific example.
See \cite[Section 3.4]{HL10} for the detail.
Let us start with a following example of terminally weighted rooted tree with legs $\ell_1,\ell_2,\ell_3$:
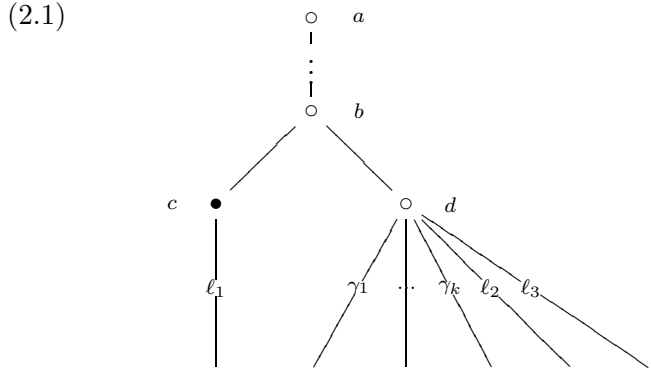
\begin{align} \label{Graph:Ex}
\xymatrix{ & & \circ \ar@{}[r]|{a}\ar@{-}[d]|{\vdots} & & & & & & & \\
 & & \circ \ar@{}[r]|{b} \ar@{-}[dl] \ar@{-}[dr] & & & & & & & \\
 & \bullet \ar@{-}[dd]|{\ell_1} \ar@{}[l]|{c} & & \circ \ar@{}[r]|{d} \ar@{-}[ddl]|{\gamma_1} \ar@{-}[dd]|{\dots} \ar@{-}[ddr]|{\gamma_k} \ar@{-}[ddrr]|{\ell_2} \ar@{-}[ddrrr]|{\ell_3} & & & & & & \\
 & & & & & & & & & \\
 & & & & & & & & & \\
}
\end{align}
If we take advancing at a {\emph{terminal}} vertex $c$, then we obtain the following tree $(1)$, where the weight of the vertex $c$ is the sum of the original weight of $c$ and the weights of trees $\gamma_1$,...,$\gamma_k$, (for short, it kills every descendant of the direct ascendant of $c$ except for legs and puts killed weights and legs on $c$). If we take advancing at a {\emph{non-terminal}} vertex $d$, then we obtain the following tree $(2)$, (for short, it shifts every other descendant of the direct ascendant of $d$ except for legs to descendant of $d$):
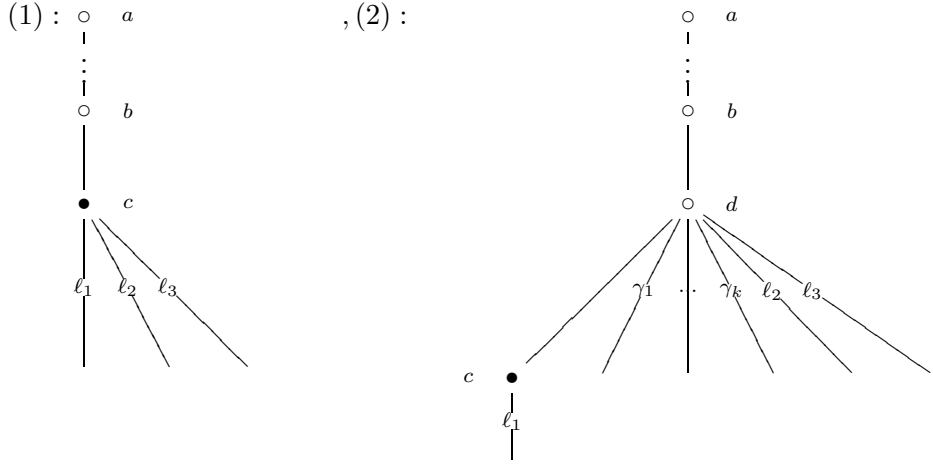
\[
(1) : \xymatrix{\circ \ar@{-}[d]|{\vdots} \ar@{}[r]|{a} & & & \\ 
\circ \ar@{-}[d] \ar@{}[r]|{b} & & &\\ 
\bullet \ar@{}[r]|{c} \ar@{-}[ddrr]|{\ell_3} \ar@{-}[dd]|{\ell_1}\ar@{-}[ddr]|{\ell_2} & & &\\ 
& & &\\
& & &
},
(2) : \xymatrix{& & & \circ \ar@{}[r]|{a}\ar@{-}[d]|{\vdots} & & & & \\
& & & \circ \ar@{}[r]|{b} \ar@{-}[d] & & & & \\
& & & \circ \ar@{}[r]|{d} \ar@{-}[ddll] \ar@{-}[ddl]|{\gamma_1} \ar@{-}[dd]|{\dots} \ar@{-}[ddr]|{\gamma_k} \ar@{-}[ddrr]|{\ell_2} \ar@{-}[ddrrr]|{\ell_3}  & & & \\
& & & & & & & \\
 & \bullet \ar@{}[l]|{c} \ar@{-}[d]|{\ell_1}& & & & & & \\
 & & & & & & &
}
\]

\subsection{Vakil-Zinger blow-up $\wtil{\fM}$ of $\fM_{1,k}$} \label{VZD}

In this section, we will describe a Vakil-Zinger's blow-up in \cite{VZ08} locally.
More precisely, for each closed point $x \in \fM^w_{1,k}$, we will describe local coordinate functions on $\wtil{\fM}$ in terms of corresponding terminally weighted rooted tree $\Gamma$ to $x$.
Let $V$ be a neighborhood of image of $x$ on $\fM_{1,k}$.

We associate a coordinate function $\tau_v: V \to \CC$ for each $v \in Ver(\Gamma)^*$ corresponding to a smoothing node (associated to an edge of $v$ and its direct ascendant).
Let $v_1, ..., v_b \in Ver(\Gamma)^*$ be direct descendants of the branch vertex.
Take a blow up 
$$V^{[1]}:= Bl_{\{\tau_{v_1}= \cdots =\tau_{v_b} =0\}} V \subset V \times \PP^{b-1}.$$
Let $y_1, ..., y_b$ be homogeneous coordinates on $\PP^{b-1}$.
Now, we have $b$ many choices of open charts $V^{[1]}_i:=\{\tau_{v_i} \neq 0\}$ of $V^{[1]}$ for each $i$.
We observe that $y_1, ..., y_{i-1}, \tau_{v_i}, y_{i+1}, ..., y_b$ are coordinate functions on $V^{[1]}_i$.
We associate an advancing graph $\Gamma^{[1]}_i$ at $v_i$ of $\Gamma$ to $V^{[1]}_i$.
We may regard $y_j$ as corresponding coordinates of direct descendants of the branch vertex of $\Gamma^{[1]}_i$ if $v_i$ is not terminal in $\Gamma$.

For instance, let us consider the graph $\Gamma$ in \eqref{Graph:Ex}.
We have a coordinate $\tau_v$ for each $v \in Ver(\Gamma)^*$.
If we do an advancing at $c$, then the corresponding coordinates for $Ver(\Gamma^{[1]}_c)^*$ are $\tau_b$ and $\tau_c$.
If we do an advancing at $d$, then the corresponding coordinates for $Ver(\Gamma^{[1]}_d)^*$ are $\tau_b$, $\tau_d$, $y_c$ and $\tau_v$ for descendants $v$ of $d$.

For simplicity, we let $(V^{[1]}, \Gamma^{[1]}) := (V^{[1]}_i, \Gamma^{[1]}_i).$
If $v_i$ is not terminal, we do this process again to get $(V^{[2]}, \Gamma^{[2]})$.
We can do this process (possibly many choices) to get $(V^{[q]}, \Gamma^{[q]})$ until $\Gamma^{[q]}$ will be a path tree.
It is known that $V^{[q]}$ is a neighborhood of $\wtil{\fM}$ and we get coordinate functions on $V^{[q]}$ corresponding to $Ver(\Gamma^{[q]})^*.$

\subsection{Local defining equations and decomposition of $\wtil{\cM}$} \label{localdefiningequations}

Recall that $\wtil{\cM} = \wtil{\fM} \times_{\fM_{1,k}} \bar{M}_{1,k}(\PP^n,d)$.
For a closed point at $\wtil{\cM}$, we can associate a terminally weighted rooted tree $\Gamma$ because we have a forgetful morphism $\wtil{\cM}  \to  \bar{M}_{1,k}(\PP^n,d) \to \fM^w_{1,k}$.
Let us do the process described in Section \S \ref{VZD} to get $(V^{[q]}, \Gamma^{[q]})$ where
$$\Gamma^{[q]} = \bar{\star v_1\dots v_r a_1\dots a_q}$$
is a path tree. 
Here, $\bar{\star v_r}$ is a trunk of $\Gamma$.

Let $\fM^{div}_{1,k}$ be the moduli stack parametrizing $(C,D)$ where $C \in \fM_{1,k}$ and $D$ is a divisor on $C$.
Let $\wtil{\fM}^{div} := \wtil{\fM} \times_{\fM_{1,k}} \fM^{div}_{1,k}$ and $\cV^{[q]} := V^{[q]} \times_{\wtil{\fM}} \wtil{\fM}^{div}$.
We obtain local coordinate functions $\wtil{\tau}_{v_1}, ..., \wtil{\tau}_{v_r}, \wtil{\tau}_{a_1}, ..., \wtil{\tau}_{a_q}$ on $\cV^{[q]}$.
Let $w_1$, ..., $w_n$ be coordinates on $\CC^n$ and let $\wtil{\tau} := \wtil{\tau}_{v_1} \cdots \wtil{\tau}_{v_r} \wtil{\tau}_{a_1} \cdots \wtil{\tau}_{a_q}$.
\begin{prop}\cite[Theorem 2.17 and Theorem 2.19]{HL10} \label{LocalEqF}
Locally, $\wtil{\cM}$ is embedded as an open substack of $F^{-1}(0)$ where 
\begin{align*} 
F: \cV^{[q]} \times \CC^n \times \CC^{dn} & \to \CC^n \\
(\wtil{\tau}_{v_1}, ..., \wtil{\tau}_{v_r}, \wtil{\tau}_{a_1}, ..., \wtil{\tau}_{a_q}, w_1, ..., w_n) & \mapsto (\wtil{\tau}w_1, ..., \wtil{\tau}w_n). \nonumber
\end{align*}
\end{prop}

Now, we want to describe an irreducible decomposition of $\wtil{\cM}$.
We note that the number of legs on $\Gamma^{[j]}$ is $k$ and the sum of weights on $\Gamma^{[j]}$ is $d$ for all $j=0, 1, ..., q$ where $\Gamma^{[0]} := \Gamma$.
\begin{const}\label{assignment}
We assign each element of $Ver (\Gamma^{[q]})^*$ to a set of pairs $\{(d_1,L_1),\dots,(d_{\ell},L_{\ell})\}$ where $d_1 + \dots + d_{\ell}=d$ and $L_i \subset [k]$ are sets of legs disjoint to each others. 
For a vertex $v_i$, we assign a set of single pair $\{(d,L)\}$ where $L$ is a set of legs in $\Gamma$ attached on vertices $v_i,\dots,v_r,a_1,\dots,a_q$. 
For a vertex $a_j$, consider the pruning at direct descent vertices of the branch vertex of $\Gamma^{[j-1]}$. 
Then we have degrees $d_1,\dots,d_{\ell_j}$ and sets of legs $L_1,\dots,L_{\ell_j}$ at direct descent vertices of the branch vertex of the pruning of $\Gamma^{[j-1]}$. 
We assign a set of pairs $\{(d_1,L_1),(d_2,L_2),\dots,(d_{\ell_j,},L_{\ell_j})\}$ for the vertex $a_j$. 
Here, the order of pairs is not necessary.
\end{const}

Let $\fS$ be a set of pairs $\{(d_1,L_1),\dots,(d_{\ell},L_{\ell})\}$ appeared in Construction \ref{assignment}. 
For $\mu  \in \fS$, we assign an irreducible component of $\wtil{\cM}$, 
\[\wtil{\cM}^\mu := \wtil{\cM} \times_{\bar{M}_{1,k}(\PP^n,d)}\bar{M}_{1,k}(\PP^n,d)_\mu;\]
see \cite{VZ08} for the definition of $\bar{M}_{1,k}(\PP^n,d)_\mu$. %More precisely, $\wtil{M}_{1,k}
By \cite{VZ08}, we have an irreducible decomposition
\begin{equation*}\label{globaldecomposition0}
\wtil{\cM}=\wtil{\cM}^{red} \cup \left( \bigcup\limits_{\mu \in \fS} \wtil{\cM}^\mu \right).
\end{equation*}

Let $s: \{v_1, ..., v_r, a_1, ..., a_q\} \to \fS$ be an assignment appeared in Construction \ref{assignment}.
The decomposition 
\begin{equation}\label{localdecomposition0}
\{F=0\}=\left\{w_1=\dots=w_n=0\right\}\cup \left( \bigcup\limits_i \left\{ \wtil{\tau}_{v_i}=0 \right\}  \cup \bigcup\limits_j\left\{\wtil{\tau}_{a_j}=0\right\}\right)
\end{equation}
is a local description of
\begin{equation*}
\wtil{\cM}^{red} \cup \left( \bigcup\limits_{i} \wtil{\cM}^{s(v_i)} \cup \bigcup\limits_{j} \wtil{\cM}^{s(a_j)}  \right).
\end{equation*}
It contains the following meanings: 
\begin{itemize}
\item if $\mu \notin \text{im}s$, $\{F = 0\} \cap \wtil{\cM}^\mu =\varnothing,$
\item locally, $\left\{ \wtil{\tau}_{v_i}=0 \right\}$, $\left\{ \wtil{\tau}_{a_j}=0 \right\}$ coincide with $\wtil{\cM}^{s(v_i)}$, $\wtil{\cM}^{s(a_j)}$,
\item locally, $\{w_1=\dots=w_n=0\}$ coincides with $\wtil{\cM}^{red}$.
\end{itemize}
In particular, $\wtil{\cM}^{red}$ is smooth; $\wtil{\cM}^{rat} := \bigcup\limits_{\mu \in \fS} \wtil{\cM}^\mu$ has normal crossing singularities and $\wtil{\cM}^{red}$ and $\wtil{\cM}^{rat}$ meet transversally.

Let $\pi : \cC \to \wtil{\cM}$ be the universal curve and 
$ev : \cC \to \PP^n$ be the evaluation morphism.
In a similar way as in \cite[Proposition 3.2, Lemma 3.4]{CL15}, we can prove the following.

\begin{prop} \label{Decomp:Sh}
On local chart $U := \{F=0\}$ in \eqref{localdecomposition0}, we have 
\begin{align} \label{Dec:Sheaves}
\RR \pi_{*}ev^*\cO_{\PP^n}(r) & \stackrel{qis}{\cong} [\cO_{U} \xrightarrow{\wtil{\tau}} \cO_{U}  ] \oplus [\cO_{U}^{\oplus rd}  \to 0 ], \\
\RR \pi_{*}(ev^*\cO_{\PP^n}(-r) \otimes \omega_{\pi} ) & \stackrel{qis}{\cong} [\cO_{U} \xrightarrow{\wtil{\tau}} \cO_{U}  ] \oplus [  0 \to \cO_{U}^{\oplus rd} ],  \nonumber
\end{align}
for $r \in \ZZ_{>0}$.
\end{prop}
We note that two equalities in \eqref{Dec:Sheaves} are (Serre) dual to each other.
For every $r \geq 0$, $\pi_{\mu *}ev_{\mu}^*\cO_{\PP^n}(r)$ is a locally free sheaf over $\wtil{\cM}^{\mu}$ and $\pi_{red *}ev_{red}^*\cO_{\PP^n}(r)$ is a locally free sheaf over $\wtil{\cM}^{red}$
where $\pi_\mu : \cC \to \wtil{\cM}^{\mu}$ is the universal curve and $ev_\mu : \cC \to \PP^n$ is the evaluation morphism. The morphisms $\pi_{red}$ and $ev_{red}$ are similarly defined.

\medskip

\section{Structure of desingularization of stable map with fields} \label{Str:P}

Recall that $\wtil{\cM}^p:= \wtil{\cM} \times_{\bar{M}_{1,k}(\PP^n,d)} \bar{M}_{1,k}(\PP^n,d)^p$.
Here, the moduli space $\bar{M}_{1,k}(\PP^n,d)^p$ can be viewed as follows: let $[A \xrightarrow{d} B]$ be a two term presentation by locally free sheaves of
$$\oplus_i \RR \pi_{*}(ev^*\cO_{\PP^n}(-\deg f_i) \otimes \omega_{\pi} ) $$ %\cong \bigoplus_i [\cO_{U} \xrightarrow{\wtil{\tau}} \cO_{U}  ] \oplus [  0 \to \cO_{U}^{\oplus rd} ],$$
on $\bar{M}_{1,k}(\PP^n,d)$. 
Let $p:  |A| \to \bar{M}_{1,k}(\PP^n,d)$ be a projection where $|A|$ denotes the total space of $A$.
Then, $\bar{M}_{1,k}(\PP^n,d)^p$ is a zero of a section $p^*d : \cO_{|A|} \to p^*B$.
The following proposition follows immediately from the base change theorem, Proposition \ref{LocalEqF} and Proposition \ref{Decomp:Sh}.

\begin{prop}\label{localchart}
Locally, $\wtil{\cM}^p$ is embedded as an open substack of $\wtil{F}^{-1}(0)$ where 
\begin{align*}
\wtil{F}: \cV^{[q]} \times \CC^n \times \CC^{m} \times \CC^{nd} & \to \CC^{n+m} \\
(\wtil{\tau}_{v_1}, ..., \wtil{\tau}_{v_r}, \wtil{\tau}_{a_1}, ..., \wtil{\tau}_{a_q}, w_1, ..., w_n, t_1, ..., t_m) & \mapsto (\wtil{\tau}w_1, ..., \wtil{\tau}w_n, \wtil{\tau}t_1, ..., \wtil{\tau}t_m). 
\end{align*}
Here, $t_1, ..., t_m$ are coordinates on $\CC^m$.
Moreover, an inclusion $\wtil{\cM} \hookrightarrow \wtil{\cM}^p$ by a zero section corresponds to $\{t_1=\dots = t_m=0 \} \subset \{ \wtil{F}=0\}$.
\end{prop}

%Let $l_i:=\deg f_i$ where $f_1,\dots,f_m$ are defining equations of $Q\subset \PP^n$. 
The complex 
\begin{align} \label{POT:pfield}
\RR \pi_{*}ev^*\cO_{\PP^n}(1)^{\oplus (n+1)} \bigoplus \oplus_i \RR \pi_{*}(ev^*\cO_{\PP^n}(-\deg f_i) \otimes \omega_{\pi} )
\end{align}
on $\bar{M}_{1,k}(\PP^n,d)^p$ is the dual of the natural perfect obstruction theory relative to $\fB:= \fB un^{1,k}_{\CC^*}$.
A closed point $x$ in $\bar{M}_{1,k}(\PP^n,d)^p$ has a data $(u=(u_0, ..., u_n),p=(p_1, ..., p_m))$ which is a section in
$$\RR^0 \pi_{*}ev^*\cO_{\PP^n}(1)^{\oplus (n+1)} \bigoplus \oplus_i \RR^0 \pi_{*}(ev^*\cO_{\PP^n}(-\deg f_i) \otimes \omega_{\pi} ).$$ 
A linear morphism at $x$
\begin{align} \label{CoSection}
& (\RR^1 \pi_{*}ev^*\cO_{\PP^n}(1)^{\oplus (n+1)} \bigoplus \oplus_i \RR^1 \pi_{*}(ev^*\cO_{\PP^n}(-\deg f_i) \otimes \omega_{\pi} ) ) |_x \to \CC \\
& (u'=(u'_0, ..., u'_n)  ,p'=(p'_1, ..., p'_m)) \mapsto \sum_j p'_j f_j(u) + \sum_{i,j} p_j u'_i(\partial_{u_i}f_j)(u) \nonumber
\end{align}
%defined by sending $(u'=(u'_0, ..., u'_n)  ,p'=(p'_1, ..., p'_m))  $ to 
%$$\sum_j p'_j f_j(u) + \sum_{i,j} p_j u'_i(\partial_{u_i}f_j)(u)\in \CC$$
defines a cosection on the obstruction sheaf.
By cosection localization method \cite{KL13cosec}, we can define a cycle $[\bar{M}_{1,k}(\PP^n,d)^p]\virtloc \in A_*(\bar{M}_{1,k}(Q,d)) $.
A pull-back of the complex \eqref{POT:pfield} to $\wtil{\cM}^p$ is again a dual of a perfect obstruction theory relative to $\wtil{\fB}:= \fB \times_{\fM_{1,k}} \wtil{\fM}$ by \cite[Proposition 7.2]{BF97}.
A pull-back of the cosection defines a class $[\wtil{\cM}^p]\virtloc \in A_*(\wtil{\cM}_Q) $ where $\wtil{\cM}_{Q} := \bar{M}_{1,k}(Q,d) \times_{\bar{M}_{1,k}(\PP^n,d)} \wtil{\cM}$.

\begin{lemm} \label{VirPuFor}
The equation \eqref{VirPushFor} holds true, i.e., we have
$$b_*[\wtil{\cM}^{p}]\virtloc=[\bar{M}_{1,k}(\PP^n,d)^p]\virtloc .$$
\end{lemm}
\begin{proof}
We have the following equalities:
\begin{align*}
[\bar{M}_{1,k}(\PP^n,d)^{\vec{p}}]\virtloc &=(-1)^{d\sum_i \deg f_i } [\bar{M}_{1,k}(Q,d)]\virt \\
& = (-1)^{d\sum_i \deg f_i }  b_*[\wtil{\cM}_{Q}]\virt \\
&= b_*[\wtil{\cM}^p]\virtloc
\end{align*}
%where $\wtil{\cM}_{Q} := \bar{M}_{1,k}(Q,d) \times_{\bar{M}_{1,k}(\PP^n,d)} \wtil{\cM}$.
Here, the first and third equalities come from \cite{KO18localized, CL18inv} and the second equality comes from \cite[Theorem 5.0.1]{Cos06higher}.
\end{proof}

Now, we discuss a decomposition of the intrinsic normal cone $\fC_{\wtil{\cM}^p/\wtil{\fB}}$.
%the intrinsic normal cones relative to $\wtil{\fB}$.
Let $\wtil{\cM}^{p,rat}:= \wtil{\cM}^p \times_{\wtil{\cM}} \wtil{\cM}^{rat}$. 
We note that $\wtil{\cM}^{p,rat} = \{\wtil{\tau} =0 \} \subset \{\wtil{F}=0\}$ locally.
Let $\wtil{\cM}^{p,red}$ be the gluing of $\{w_1=\cdots=w_n =t_1 = \cdots =t_m =0 \} $.
We note that $\wtil{\cM}^{p,red}  \cong \wtil{\cM}^{red}$ and $\wtil{\cM}^{p,red}\neq \wtil{\cM}^p \times_{\wtil{\cM}} \wtil{\cM}^{red}$.
Let $\wtil{\cM}^p_{red} := \wtil{\cM}^p \times_{\wtil{\cM}} \wtil{\cM}^{red}$.
Consider an intrinsic normal cone $\fC^p:= \fC_{\wtil{\cM}^{p}/\wtil{\fB}}.$
On $\wtil{\cM}^{p} \setminus \wtil{\cM}^{p,rat}$, we can see that $\fC^p|_{\wtil{\cM}^{p} \setminus \wtil{\cM}^{p,rat}} = [(\wtil{\cM}^{p} \setminus \wtil{\cM}^{p,rat}) / T_{(\wtil{\cM}^{p} \setminus \wtil{\cM}^{p,rat})/\wtil{\fB}}]$ where $T$ stands for the tangent bundle.
Indeed, 
%on $\wtil{\cM}^{p} \setminus \wtil{\cM}^{p,rat}$, $\wtil{\cM}^p = \wtil{\cM}$
$\wtil{\cM}^{p} \setminus \wtil{\cM}^{p,rat} \cong \wtil{\cM}\setminus \wtil{\cM}^{rat}$ is smooth over $\wtil{\fB}$.
Thus its closure in $\fC^p$ is $[\wtil{\cM}^{red} / T_{\wtil{\cM}^{red}/\wtil{\fB}}]=[\wtil{\cM}^{p,red} / T_{\wtil{\cM}^{p,red}/\wtil{\fB}}]$. Let
\begin{align} \label{Cone:RED}
\fC^{p,red}:=[\wtil{\cM}^{red} / T_{\wtil{\cM}^{red}/\wtil{\fB}}]=[\wtil{\cM}^{p,red} / T_{\wtil{\cM}^{p,red}/\wtil{\fB}}] \subset \fC^p.
\end{align}

\begin{prop}\label{conestr}
On $\wtil{\cM}^{p} \setminus \wtil{\cM}^{p}_{red}$, the tangent sheaf $T_{\wtil{\cM}^p / \wtil{\fB}}$ is locally free and 
$$\fC^p|_{\wtil{\cM}^{p} \setminus \wtil{\cM}^{p}_{red}} = [S / T_{( \wtil{\cM}^p \setminus \wtil{\cM}^{p}_{red} )/ \wtil{\fB}}]$$
for some rank two bundle $S$ on $\wtil{\cM}^{p} \setminus \wtil{\cM}^{p}_{red}$.
\end{prop}

\begin{proof}
The projection $pr : \wtil{\cM}^{p,rat} \to \wtil{\fB}$ is smooth onto its image by Proposition \ref{localchart} and Proposition \ref{Decomp:Sh}.
We will show the claim that the image $pr(\wtil{\cM}^{p,rat}) \subset \wtil{\fB}$ is a codimension 2 regular embedding. 
If the claim is true, then we have
\begin{align*}
&\fC_{\wtil{\cM}^{p}/\wtil{\fB}} |_{\wtil{\cM}^{p} \setminus \wtil{\cM}^{p}_{red}} \\
&\cong \fC_{\wtil{\cM}^{p,rat}/pr(\wtil{\cM}^{p,rat}) }  \times_{\wtil{\cM}^{p,rat}} pr^* N_{pr(\wtil{\cM}^{p,rat}) / \wtil{\fB}}|_{\wtil{\cM}^{p,rat} \setminus (\wtil{\cM}^{p,rat} \cap \wtil{\cM}^{p,red})} \\
& \cong [\wtil{\cM}^{p,rat} / T_{\wtil{\cM}^{p,rat}/pr(\wtil{\cM}^{p,rat}) } ]  \times_{\wtil{\cM}^{p,rat}} pr^* N_{pr(\wtil{\cM}^{p,rat}) / \wtil{\fB}} |_{\wtil{\cM}^{p,rat} \setminus (\wtil{\cM}^{p,rat} \cap \wtil{\cM}^{p,red})} \\
& \cong  [ pr^* N_{pr(\wtil{\cM}^{p,rat}) / \wtil{\fB}} / T_{\wtil{\cM}^{p,rat}/pr(\wtil{\cM}^{p,rat}) } ] |_{\wtil{\cM}^{p,rat} \setminus (\wtil{\cM}^{p,rat} \cap \wtil{\cM}^{p,red})}
\end{align*}
where $\fC$, $N$ and $T$ stand for the intrinsic normal cone, normal bundle and tangent bundle respectively.
Here, the first equivalence holds by the claim and the second equivalence comes from the smoothness of $\pi$.
So, it is enough to prove the claim.

Consider a local neighborhood $M$ of $\wtil{\fM}$.
Let $\cV := M \times_{\wtil{\fM}} \wtil{\fM}^{div}$.
In Section \S \ref{hulilocaleq}, we defined coordinate functions $\tau, \tau_{\mu}\in H^0(M,\cO_M)$ by shrinking $M$ if necessary such that their pull-backs to $\cV$ are $\wtil{\tau}$, $\wtil{\tau}_\mu \in H^0(\cV,\cO_\cV)$. 
Let $B := M \times_{\wtil{\fM}} \wtil{\fB}$.
We see that $B \times_M \{\tau=0\} \subset B$ is a regular embedding of codimension 1 and that the image of $pr$ is contained in $B \times_M \{\tau=0\}$ (locally).
Let $I$ be the (local) image of $pr$ in $B$.
%The claim is true if it is regularly embedded with codimension 2.

Let $U := \mathbf{Pic}_{\cC_M/M}$ be a relative Picard scheme of the universal curve $\cC_M$ over $M$ parametrizing $(C,L)$ where $C \in M$ and $L$ is a line bundle on $C$.
There is a smooth surjective morphism $U \to B$ of dimension 1.
Indeed, it is known that $B \cong [U/ \CC^*]$.
It is enough to show that $U \times_B I$ is embedded in $U \times_B (B \times_M \{\tau=0\} )$ of codimension 1.
Let $pr' : \wtil{\cM}^{p,rat} \to \wtil{\fB} \to \wtil{\fM}$ be the projection.
We observe that locally the image $pr' (\wtil{\cM}^{p,rat})$ is equal to $\{\tau = 0\} \subset M$. 
Since the projection $U\times_M \{\tau=0\} \to \{\tau=0\}$ is smooth of relative dimension 1, it is enough to show that there is a section $s :\{\tau =0\} \to U\times_M \{\tau=0\}$ such that 
\begin{align} \label{Pi:Pullback}
pr |_{\wtil{\cM}^{p,rat} \times_{ \wtil{\fM}  }  \{\tau=0\} }= s \circ pr' |_{ \wtil{\cM}^{p,rat} \times_{ \wtil{\fM}  }  \{\tau=0\} }
\end{align}
where $p: \wtil{\cM}^{p,rat} \times_{ \wtil{\fM}  }  \{\tau=0\} \to U \times_M \{\tau=0\}$ is a forgetful morphism.
Indeed, the image of $p$ is $U \times_B I$.

For each $\mu \in \fS$, 
we consider projections $pr_{\mu} : \wtil{\cM}^{p, \mu} \to \wtil{\fB}$, $pr'_{\mu} : \wtil{\cM}^{p, \mu} \to \wtil{\fB} \to \wtil{\fM}$ where $\wtil{\cM}^{p, \mu} := \wtil{\cM}^\mu \times_{\wtil{\cM}} \wtil{\cM}^p$.
We can decompose a universal curve $\cC_{\{\tau_{\mu}=0\}}$ on $ \{\tau_{\mu}=0\}$ into two closed subschemes $\cC'$ and $\cC''$. 
For each closed point $x$, $(\cC')_x$ is a genus one subcurve of $(\cC_{\{\tau_{\mu}=0\}})_x$, and $(\cC'')_x$ are $\ell$ rational curves attached to $(\cC')_x$. We note that this genus one subcurve $(\cC_{\{\tau_{\mu}=0\}})_x$ may have rational tails. Moreover, $\cC' \cap \cC''$ is $\ell$ many sections of $ \{\tau_{\mu}=0\}$, whose fiber over $x$ are $\ell$ attaching nodes.

Next we construct a family of line bundles on $\cC_{\{\tau_{\mu}=0\}}$ in the following way. 
By considering $\ell$ components $\cC''_1,\dots,\cC''_\ell$ of $\cC''$, we can induce maps $\phi_i :  \{\tau_{\mu}=0\} \to \fM_{0,k_i+1}^w  $. 
Let $\fB_{0,k,w}:= \fB un^{0,k}_{\CC^*, w}$ be a stack parametrizing $(C, L)$ where $C$ is a genus $0$, prestable curve with weight $w$ and $k$ marked points; and $L$ is a line bundle on $C$ whose degree is $w$.
We note that $\fB_{0,k,w}$ is a substack of $\fB un^{0,k}_{\CC^*}$ and that there is the natural morphism $\mathrm{pr}_i : \fM^w_{0,k_i+1} \to \fB_{0,k_i+1,w}$ because $g=0$. 
Let $\bar{\phi_i} := \mathrm{pr}_i \circ \phi_i :  \{\tau_{\mu}=0\} \to \fB_{0,k_i+1,w}$.
Since $\phi_i^* \cC_{\fM_{0,k_i +1}^w} \cong \cC''_i$, we obtain a line bundle $\bar{\phi}_i^*(\cL_i)$ on $\cC''_i$ where $\cL_i$ is the universal line bundle on the universal curve $\cC_{\fB_{0,k_i+1,w}}$ of $\fB_{0,k_i+1,w}$.
By twisting a suitable bundle $\pi_i^*L_i$ for some line bundle $L_i$ on $ \{\tau_{\mu}=0\} $ and the projection $\pi_i : \cC_i'' \to \{\tau_{\mu}=0\} $, we may assume that each $\bar{\phi_i}^*(\cL_i)$ is trivial when restricted on $\cC'\cap \cC''$, which is the $\ell$ sections of $ \{\tau_{\mu}=0\} $. 
Therefore we can glue trivial bundle on $\cC'$ and $\ell$ line bundles $\bar{\phi_i}^*\cL_i$ on $\cC''_i$ to a family of line bundle $\cL$ on $\cC_{\{\tau_\mu =0\}}$. 
This family of line bundles induces a section $s_{\mu} : \{ \tau_\mu=0\} \to U\times_M \{ \tau_\mu=0\}$. 
Moreover, we can easily see that sections $s_{\mu} : \{ \tau_\mu=0\} \to U\times_M \{ \tau_\mu=0\}$ glue to a section $s : \{ \tau=0\} \to U\times_M \{ \tau=0\}$.

The remaining is to check \eqref{Pi:Pullback}. 
It is enough to check it for each $\mu \in \fS$.
Since the morphism $\wtil{\cM}^{p,\mu} \times_{ \wtil{\fM}  }  \{\tau_\mu=0\}   \to \{ \tau_\mu =0\}$ is surjective, it is enough to show that if $(C,{\bf{x}}, L,u,p)$ and $(C', {\bf{x}}' , L', u', p')$ are objects in $\wtil{\cM}^{p,\mu} \times_{ \wtil{\fM}  }  \{\tau_\mu=0\}$ such that whose images in $\{\tau_\mu =0\}$ are isomorphic to each other, then they are isomorphic to each other as objects in $U \times_M \{\tau_\mu =0\}$.
Since there is a section $u$ of $L^{\oplus n+1}$ (and $u'$ of $L'^{\oplus n+1}$ respectively) and we are working on $\{\tau_\mu =0\}$, both $L$ and $L'$ are isomorphic to trivial bundles on elliptic subcurves.
Thus, $L$ and $L'$ are isomorphic under the isomorphism of curves.
This proves the claim.
\end{proof}

For each $\mu \in \fS$, let $\fC^{p,\mu}\subset \fC^p$ be the closure of the rank 2 subbundle stack $\fC^p |_{\wtil{\cM}^{p,\mu} \setminus \wtil{\cM}^{p,\mu}   \cap \wtil{\cM}^{p}_{red}}$. 
We have a decomposition of $\fC^{p}$
$$\fC^p = \fC^{p,red} \bigcup \cup_{\mu \in \fS} \fC^{p,\mu} \bigcup \fC^p_{\Delta}$$
where $\fC^p_\Delta$ is a (not necessarily irreducible) component in $\fC^p$ supported on $\Delta^p:= \wtil{\cM}^{p}_{red} \cap \wtil{\cM}^{p,rat}$.
The cone $\fC^p_\Delta$ can also be decomposed into $\fC^p_\Delta = \bigcup_{\mu \in \fS} \fC^{p,\mu}_\Delta$ where $\fC^{p,\mu}_\Delta$ is a component supported on $\wtil{\cM}^{p}_{red} \cap \wtil{\cM}^{p,\mu}$.

Let $E^p$ be a complex defined by pull-back of \eqref{POT:pfield} to $\wtil{\cM}^p$.
Let $\sigma$ be a pull-back cosection on $h^1(E^p)$ (obstruction sheaf).
On the other hand, by Lemma \ref{VirPuFor} and projection formula, we have

\begin{align*}%\label{conedecomposition}
&(-1)^{d\sum\limits^m_{i=1}\deg f_i} GW_{1,d}(\alpha)\\
&=(-1)^{d\sum\limits^m_{i=1}\deg f_i} [\bar{M}_{1,k}(Q,d)]\virt \cap \prod\limits^k_{i=1} ev_i^*(\alpha_i) \nonumber \\
& = \Gysin{\bdst{E^p},\sigma}[\fC^{p, red}]\cap \prod\limits^k_{i=1} ev_i^*(\alpha_i) + \sum\limits_{\mu \in \fS} \Gysin{\bdst{E^p},\sigma}[\fC^{\mu}] \cap \prod\limits^k_{i=1}ev_i^*(\alpha_i) \nonumber
\end{align*}
where $\Gysin{\bdst{E^p},\sigma}$ is a localized Gysin map which is defined in \cite{KL13cosec} and $[\fC^\mu]=[\fC^{p,\mu}]+[\fC^{p,\mu}_\Delta]$.
In Section \S 1.2, we define
\begin{align} \label{DEF:A}
A^{red} := \Gysin{\bdst{E^p},\sigma}[\fC^{p, red}], \ \ A^{rat} := \sum\limits_{\mu \in \fS} \Gysin{\bdst{E^p},\sigma}[\fC^{\mu}] .
\end{align}

\medskip

\section{Contribution from reduced part: Proof of \eqref{A:red}}\label{redcont}

In this section, we want to show the following proposition:
\begin{prop} \label{conedecomposition:RED}
We have
$$(-1)^{d\sum\limits^m_{i=1}\deg f_i}0^!_{N^{red}}[C_{\wtil{\cM}^{red}_{Q} / \wtil{\cM}^{red} }] = \Gysin{\bdst{E^p},\sigma}[\fC^{p, red}] .$$
\end{prop}
Proposition \ref{conedecomposition:RED} implies \eqref{A:red} by Definition \ref{Reduced::Class} and \eqref{DEF:A}.
%$$(-1)^{d\sum\limits^m_{i=1}\deg f_i}[\bar{M}^{red}_{1,k}(Q,d)]\virt = b'_*A^{red}.$$
On $\wtil{\cM}^{p,red} = \wtil{\cM}^{red}$, let 
%%%%%%%%%%%%%%%%%%%%%%%%%%%%%%%%%%%%%%%%%%%%여기부터%%%%%%%%%%%%%%%%%%%%%%%%%%%%%%%%%%%%%%%%%%%%%%%%%%%%%%%%%%%%%%%%%%%%%%%%%%%%%%%%%%
%In Section \ref{conedecomp}, we defined $\fC_{red}$ as the closure of the image of the zero section of $(\wtil{\cM}^{\vec{p}} \setminus \wtil{\cM}_{rat}^{\vec{p}})$ in $\bdst{E_{\wtil{\cM}^{\vec{p}}/\wtil{\fB}}\dual}$. In this secction, we identify $\wtil{\cM}_{\red} \sub \wtil{\cM}_{red}^{\vec{p}}$ via the zero section embedding $\wtil{\cM} \hra \wtil{\cM}^{\vec{p}}=|\pi_{\wtil{\cM}*} \cP_{\wtil{\cM}}|$. By Proposition \ref{localchart}, we have $\wtil{\cM}\setminus \wtil{\cM}_{rat} \subset \wtil{\cM}_{red} = \wtil{\cM}^{\vec{p}} \setminus \wtil{\cM}^{\vec{p}}_{rat}$. 
%Before more explicit computation of $\fC^{p,red}$, we introduce some notations
\begin{align*}
&\EE_1 := \RR\pi_{*} ev^* \cO_{\PP^n}(1)^{\oplus(n+1)},
\EE_2 := \RR\pi_{*}\left(\bigoplus\limits_{i=1}^m ev^* \cO_{\PP^n}(-\deg f_i) \otimes \omega_{\pi}\right), \\
&V_1^p:=\RR^1\pi_{*} ev^* \cO_{\PP^n}(1)^{\oplus(n+1)}, V_2^p:= \RR^1\pi_{*}\left(\bigoplus\limits_{i=1}^m ev^* \cO_{\PP^n}(-\deg f_i) \otimes \omega_{\pi}\right).
\end{align*}
Note that $\wtil{\cM}^{p,red}=\wtil{\cM}^{red}$ is smooth and irreducible. Moreover we have $V_1^{p}|_{\wtil{\cM}\setminus \wtil{\cM}_{rat}} = 0$ and $\pi_{*}\left(\bigoplus\limits_{i=1}^m ev^* \cO_{\PP^n}(-\deg f_i) \otimes \omega_{\pi}\right)|_{\wtil{\cM}\setminus \wtil{\cM}_{rat}}=0$. Therefore we have $\fC^{p,red} = \bdst{\EE_1}\times_{\wtil{\cM}^{red}} \bar{0}_{\bdst{\EE_2}}$. Here, a zero section of a vector bundle stack is defined by the following. Let $\GG$ be a vector bundle stack on $X$, which is locally represented by a quotient $[G^1/G^0]$ of a two term complex of vector bundles, $[G^0 \stackrel{\phi}{\to} G^1]$. The zero section $0_{\GG} : X \to \GG$ is locally defined by $[G^0/G^0] \to [G^1/G^0]$ which is induced from a morphism of complex $[G^0 \stackrel{id}{\to} G^0] \stackrel{(id,\phi)}{\to} [G^0 \stackrel{\phi}{\to} G^0]$. By $\bar{0}_\GG$, we mean a closure of an image of the zero section $0_{\GG}$, which is equal to $[\bar{\image\phi}/G^0] \subset [G^1/G^0]$.

%On the other hand, it is well-known that the relative perfect obstruction theory $E_{\wtil{\cM}^{p}/\wtil{\fB}}$ is quasi-isomorphic to the 2-term locally free resolution $[ F_{-1} \stackrel{\rou}{\to} F_0 ]$. We let $F_{-1}\dual := F^1$ and $F_{0}\dual := F^0$. Then we obtain $E\dual_{\wtil{\cM}^{p}/\wtil{\fB}} = [F^0 \stackrel{d}{\to} F^1 ]$.
%Then, zero section morphism $\wtil{\cM}^{p} \to \bdst{E_{\wtil{\cM}^{p}/\wtil{\fB}}\dual}$ is induced from the morphism of complexes $[F^0 \stackrel{id}{\to} F^0 ] \stackrel{(id,d)}{\to} [F^0  \stackrel{d}{\to} F^1 ]$ by taking the bundle stack functor $h^1/h^0$. 
Then, we consider the closure of the image $\image d \subset F^1$ in $F^1$. Then we define $[\bar{\image d}/F^0]$ as a closure of the zero section of the bundle stack $\bdst{E_{\wtil{\cM}^{p}/\wtil{\fB}}\dual}$ and denote it by $\bar{0}_{\bdst{E_{\wtil{\cM}^{p}/\wtil{\fB}}\dual}}$.

Next, we prove the following lemma, which is the modification of \cite[Lemma 5.3]{CL15}.
\begin{lemm}\label{technicallemma}
Let $G\hseq = [G_{-1} \stackrel{\rou}{\to} G_{0}]$ be a $2$-term complex of locally free sheaves on an integral Deligne-Mumford stack $X$. Let $G^0 =: G_0\dual$, $G^1 := (G_{-1})\dual$ and $(G\hseq)\dual =: G\cseq = [G^0 \stackrel{d}{\to} G^1]$.
We define a bundle stack $\GG:=\bdst{G\cseq} = [G^1/G^0]$. We assume that there is a cosection $\eta : h^1(G\cseq) \to \cO_{M}$. Suppose that $h^0(G\hseq)$ is a torsion sheaf and the image sheaf $\image \rou$ is locally free. Then $h^{-1}(G\hseq)$ is locally free and there is a surjective morphism of sheaves $ h^1(G\cseq) \to h^{-1}(G\hseq)\dual$ whose dual is an isomorphism. Moreover, we have $\bar{0}_{\GG} \subset \GG(\eta)$ and
\[
\Gysin{\GG,\eta}[\bar{0}_{\GG}] = \Gysin{h^{-1}(G\hseq)\dual,\eta}[0_{h^{-1}(G\hseq)\dual}]
\]
where second $\eta$ is the cosection on $h^{-1}(G\hseq)\dual$ induced by an isomorphism $ \Hom( h^{-1}(G\hseq)\dual, \cO) \to \Hom( h^1(G\cseq), \cO) $.
\end{lemm}
\begin{proof}
Consider the following exact sequences
\[
0 \to h^{-1}(G\hseq) \stackrel{a}{\to} G_{-1} \stackrel{\rou}{\to} G_0 \stackrel{b}{\to} h^0(G\hseq) \to 0.
\]
Since $\image \partial$ is locally free, $h^{-1}(G\hseq)$ is locally free. 
We have two exact sequences
\begin{align*}
0 \to h^0(G\cseq) \to G^0 \stackrel{d}{\to} G^1 \to h^1(G\cseq) \to & 0 \\
0 \to \image \partial \dual \stackrel{\iota}{\to} G^1 \stackrel{a \dual}{\to} h^{-1}(G\hseq)\dual \to & 0.
\end{align*}
Since the morphism $G^0 / h^0(G\cseq) \to \image \partial \dual $ induced by $(\image \rou \hra G_0)\dual$ is injective, $ h^1(G\cseq) \to h^{-1}(G\hseq)\dual$ is a surjective morphism of sheaves.
Since the kernel of $d \dual = \partial$ and the kernel of $\iota \dual$ are identical, we have an isomorphism $ \Hom( h^{-1}(G\hseq)\dual, \cO) \to \Hom( h^1(G\cseq), \cO) $.
Hence, we may consider the cosection $\eta$ of $h^1(G\cseq)$ as a cosection of $h^{-1}(G\hseq)\dual$.
By definition, $\bar{0}_{\GG}=[\bar{\image d}/G^0]$. Since $h^0(G\hseq)$ is a torsion sheaf and $X$ is integral, we can choose an open dense subset $U \subset X$ such that $h^0(G\hseq)|_U = 0$. Let us denote $\bar{\image d}$ by $Z$. Then $Z$ is equal to the closure of $Z\times_X U$ in $G^1$. Since $h^0(G\hseq)|_U=0$, we obtain the short exact sequence of locally free sheaves
\[
0 \to G^0|_U \stackrel{d}{\to} G^1|_U \stackrel{a\dual}{\to} h^1(G\cseq)|_U \to 0.
\]
Therefore, $\image d \times_X U = Z \times_X U = |\ker a\dual |_U|$. Since $\ker a\dual$ is irreducible, we conclude that $Z = |\ker a\dual|$. Therefore, if we consider a composed cosection $\eta \circ a\dual : G^1 \to \cO_X$, it is clear that $Z \subset G^1(\eta \circ a\dual)$. Thus we check that $\bar{0}_{\GG} \subset \GG(\eta)$, and $\Gysin{\GG,\eta}[\bar{0}_{\GG}]=\Gysin{G^1,\eta\circ a\dual}[Z] = \Gysin{G^1,\eta\circ a\dual}[|\ker a\dual|] = \Gysin{h^{-1}(G\hseq)\dual,\eta}[0_{h^{-1}(G\hseq)\dual}]$.
\end{proof}

As we explained in \eqref{CoSection}, we define cosections $\sigma_1 : V_1^p \to \cO_{\wtil{\cM}^{p,red}}$, $\sigma_2 : V_2^p \to \cO_{\wtil{\cM}^{p,red}}$ and $ \sigma = \sigma_1 \oplus \sigma_2 : V_1^p \oplus V_2^p \to \cO_{\wtil{\cM}^{p,red}}$.
Moreover, it is well-known that
\begin{align*}
\EE_1 \stackrel{qis}{\cong} \left[A_0 \stackrel{d_A}{\to} A_1 \right], \EE_2 \stackrel{qis}{\cong} \left[B_0 \stackrel{d_B}{\to} B_1 \right]
\end{align*}
for some vector bundles $A_0,A_1,B_0,B_1$ on $\wtil{\cM}^{p,red}$. We have $\EE_1 = [A_1/A_0]$, $\EE_2 = [B_1/B_0]$ and there are induced cosections $\sigma_1 : A_1 \to \cO_{\wtil{\cM}^{p,red}}$, $\sigma_2 : B_1 \to \cO_{\wtil{\cM}^{p,red}}$ and $\sigma : A_1\oplus B_1 \to \cO_{\wtil{\cM}^{p,red}}$. Let $Z_B : = \bar{\image d_B}$.
By Proposition \ref{Decomp:Sh}, we can easily check that the 2-term complex $[B_0 \stackrel{d_B}{\to} B_1]\dual$ satisfies the conditions of Lemma \ref{technicallemma}.

By Lemma \ref{technicallemma}, we compute the contribution from the reduced component as follows. Let $i : (\sigma\dual)^{-1}(0) \hra (\sigma_2\dual)^{-1}(0)$ be an inclusion of the degeneracy loci. We note that this is an isomorphism.
\begin{align} \label{REFEULER}
& \Gysin{\bdst{E_{\wtil{\cM}^{p}/\wtil{\fB}}\dual}, \sigma}[\fC^{p,red}] \\ \nonumber
& = i_* \Gysin{\bdst{\EE_1}\oplus \bdst{\EE_2}, \sigma}[\fC^{p,red}] \\ \nonumber
& = i_* \Gysin{\bdst{\EE_1}\oplus \bdst{\EE_2},\sigma}[\bdst{\EE_1}\times_{\wtil{\cM}^{p,red}} \bar{0}_{\bdst{\EE_2}}] \\ \nonumber
%& = i_*\Gysin{A_1\oplus B_1, \sigma_1\oplus \sigma_2}[A_1\times_{\wtil{\cM}^{p,red}} Z_B] \\ \nonumber
%& = i_* \Gysin{B_1,\sigma_2}\circ \Gysin{A_1}[A_1\times_{\wtil{\cM}^{p,red}} Z_B] \\ \nonumber
& = \Gysin{\bdst{\EE_1}\oplus \bdst{\EE_2},\sigma_2}[\bdst{\EE_1}\times_{\wtil{\cM}^{p,red}} \bar{0}_{\bdst{\EE_2}}] \\ \nonumber
& = \Gysin{N \dual,\sigma_2}[0_{N \dual}]=\Gysin{N \dual,\sigma_2}[\wtil{\cM}^{p,red}].
\end{align}
where $N:= \pi_* ev^* \oplus_i \cO_{\PP^n}(\deg f_i) \cong (V_2^p)\dual$. Here, the third equality will be explained in the proof of Proposition \ref{anallocalGysin} and the fourth equality is by Lemma \ref{technicallemma}.
We note that $(\sigma_2\dual)^{-1}(0) = \wtil{\cM}_Q^{red}$. 
%Let $\iota : (\sigma_2\dual)^{-1}(0) \hra \wtil{\cM}^{red}$ be an inclusion. 
%%%%%%%%%%%%%%%%%%%%%%%%%%%%%%%%%%%%%%%%%%%%%%%%%%%%%%%%%%%%%%%%%%%%%%%%%%%%%%%%%%%%%%%%

\begin{proof}[Proof of Proposition \ref{conedecomposition:RED}]
Let
\begin{align*}
T:= T_{\wtil{\cM}^{p,red}/ \wtil{\fB}}  \cong   \pi_* ev^* \cO_{\PP^n}(1)^{\oplus n+1}.
\end{align*}
The defining morphisms $f_1, ..., f_m$ induces a section $\beta: \cO_{\wtil{\cM}^{p,red}} \to N$.
Since $\beta^{-1}(0) = \wtil{\cM}^{red}_Q$ and $N^{red} \cong N|_{\wtil{\cM}^{red}_Q}\cong N|_{\beta^{-1}(0)} $, we have 
\begin{align} \label{RED1}
0^!_{N^{red}}[C_{\wtil{\cM}^{red}_{Q} / \wtil{\cM}^{red} }] = \beta^! [\wtil{\cM}^{red}] = \beta^! [\wtil{\cM}^{p,red}] .
\end{align}
%By base change theorem, $E^p |_{\wtil{\cM}^{p,red}} \cong [T \xrightarrow{0} N^\vee].$
We can observe that the cosection $\sigma_2 : N \dual  \to \cO_{\wtil{\cM}^{p,red}}$ is exactly equal to $\beta^\vee$.
By \eqref{REFEULER} we have 
\begin{align} \label{RED2}
\Gysin{\bdst{E^p},\sigma}[\fC^{p, red}] =\Gysin{N \dual,\sigma_2}[\wtil{\cM}^{p,red}] = 0^!_{N^\vee, \beta^\vee} [\wtil{\cM}^{p,red}].
\end{align}
%The first equality comes from \eqref{Cone:RED}.
Proposition \ref{conedecomposition:RED} is obtained by \eqref{RED1}, \eqref{RED2} and the fact
$$(-1)^{d\sum_i \deg f_i}\beta^![\wtil{\cM}^{p,red}] = 0^!_{N^\vee, \beta^\vee} [\wtil{\cM}^{p,red}].$$
We note that $\rank N= d \sum_i \deg f_i$. 
\end{proof}

\medskip

\section{Contribution from rational part: Proof of \eqref{A:rat}}\label{rationalcont}

In this section we compute $\Gysin{\bdst{E^p},\sigma}[\fC^{\mu}]$ for each $\mu \in \fS$.
We introduce some notations for convenience.
Let
\begin{align*}
V_1 := \RR^1\pi_{*} ev^* \cO_{\PP^n}(1)^{\oplus(n+1)}, \ \
V_2 := \RR^1\pi_{*}\left(\bigoplus\limits_{i=1}^m ev^* \cO_{\PP^n}(-\deg f_i) \otimes \omega_{\pi}\right),
\end{align*}
and let $V:=V_1\oplus V_2$ be vector bundles on $\wtil{\cM}^{rat}$. 
Let $V_1^p$, $V^p_2$, and $V^p$ be similarly defined vector bundles on $\wtil{\cM}^{p,rat}$.
Note that $\rank V_1 = n+1$ and $\rank V_2 = d(\sum_i \deg f_i) +m$.
Let $$A := \RR^0\pi_{*} ev^* \cO_{\PP^n}(1)^{\oplus(n+1)} \oplus
\RR^0\pi_{*}\left(\bigoplus\limits_{i=1}^m ev^* \cO_{\PP^n}(-\deg f_i) \otimes \omega_{\pi}\right)$$
be a vector bundle on $\wtil{\cM}^{p,rat}$.
Note that $E^p |_{\wtil{\cM}^{p,rat}} \cong [A \xrightarrow{0} V^p]$.
Let $\sigma_1 :  V_{1}^{p}  \to \cO_{\wtil{\cM}^{p,rat}}$ and $\sigma_2 :  V_{2}^{p}  \to \cO_{\wtil{\cM}^{p,rat}}$ be cosections defined by the pullbacks of the morphisms
\begin{align*}
%\wtil{\sigma}_1 : & V_{1}^{\vec{p}}  \to \cO_{\wtil{\cM}_{\mu}^{\vec{p}}}, \ \
& (u_0',\dots,u_n')  \mapsto \sum\limits_{i=1}^{m} p_i \sum\limits_{j=0}^{n} \frac{\rou f_i}{\rou u_j}(u_0,\dots,u_n)u_j' , \\
%\wtil{\sigma}_2 : & V_{2}^{\vec{p}}  \to \cO_{\wtil{\cM}_{\mu}^{\vec{p}}}, \ \
& (p_1',\dots,p_m')  \mapsto \sum\limits_{i=1}^{m} p_i'f_i(u_0, \dots, u_n),
\end{align*}
respectively.
Then $\sigma = \sigma_1 \oplus \sigma_2$.
Let $C^\mu := \fC^\mu \times_{\bdst{E^p}} |V^{p}|$ and $C^{p,\mu}$, $C^{p,\mu}_\Delta$ be similarly defined stacks.
We denote $| S |$ the total space of a vector bundle $S$ in general.
Note that 
\begin{align}\label{CONE:DECOMP}
\Gysin{\bdst{E^p},\sigma}[\fC^{\mu}] = 0^!_{V^p, \sigma}[C^\mu] = 0^!_{V^p, \sigma}[C^{p, \mu}] + 0^!_{V^p, \sigma}[C^{p,\mu}_\Delta]
\end{align}
for each $\mu \in \fS$.

Let $\gamma : \wtil{\cM}^{p,rat} \to \wtil{\cM}^{rat}$ be a projection and 
let $\Theta : = \wtil{\cM}^{red} \cap \wtil{\cM}^{rat}$.
The following proposition is crucial for the calculation.

\begin{prop} \cite[Proposition 7.1]{CL15} \label{conesupp}
There is a sub-bundle $F=F_1 \oplus\dots \oplus F_m \subset V_2|_{\Theta}$ where $F_1,\dots, F_m$ are line bundles such that :
\[ C^{\mu} \cap | 0\oplus V^{p}_2  | \subset  \wtil{\cM}^{p,rat} \cup |\gamma^* F|
\]
for each $\mu$. Here, $\wtil{\cM}^{p,rat}$ is the image of a zero section.
\end{prop}

Next, we recall the constructions in \cite[Section 6]{CL15}. 
Let $\PP  : =  \PP(A_2  \oplus \cO_{\wtil{\cM}^{rat}})$ be a compactification of $\wtil{\cM}^{p,rat}$ 
where $A_2:= \RR^0\pi_{*}\left(\bigoplus\limits_{i=1}^m ev^* \cO_{\PP^n}(-\deg f_i) \otimes \omega_{\pi}\right)$ is a vector bundle on $\wtil{\cM}^{rat}$.
Note that $|A_2| \cong \wtil{\cM}^{p,rat}$.
Let $\bar{\gamma} : \PP \to \wtil{\cM}^{rat}$ be the projection. 
Note that $\wtil{\cM}^{rat}$ is identified with $\PP(0\oplus \cO_{{\cM}^{rat}})\subset \PP$. 
%%%%%%%%%%%여기부터%%%%%%%%%%%%
We will call it the zero section of $\PP$.
Let $D_{\infty}:=\PP(A_2 \oplus 0)$ be the divisor at infinity. 
Let $\bar{V}^{p}_{1}:=\bar{\gamma}^*V_{1}(-D_{\infty})$ and $\bar{V}^{p}_{2}:=\bar{\gamma}^*V_{2}$. 
There is a tautological section $\bar{t}_{A_2}$ of the vector bundle $\bar{\gamma}^*A_2(D_{\infty})$ on $\PP$. 
There are canonical homomorphisms of sheaves $\xi_1 : V_{1}\otimes A_2 \to \cO_{\wtil{\cM}^{rat}}$, $\xi_2 : V_{2} \to \cO_{\wtil{\cM}^{rat}}$ defined by the pullbacks of
\begin{align*}
 (\dot{u_0},\dots,\dot{u_n})\otimes(p_1,\dots,p_m) & \mapsto \sum\limits_{i=1}^{m} p_i \sum\limits_{j=0}^{n} \frac{\rou f_i}{\rou u_j}(u_0,\dots,u_n)\dot{u_j} , \\
(\dot{p_1},\dots,\dot{p_m}) & \mapsto \sum\limits_{i=1}^{m} \dot{p_i}f_i(u_0, \dots, u_n),
\end{align*}
respectively.
We define $\bar{\sigma}_1 : \bar{V}^{p}_{1} \to \cO_{\PP}$ to be $\bar{\gamma}^*(\xi_1)(- , \bar{t}_N)$ 
and $\bar{\sigma}_2 : = \bar{\gamma}^*(\xi_2)$. 
And we define a cosection $\bar{\sigma}$ by $\bar{\sigma}_1 \oplus \bar{\sigma}_2 : \bar{V}^{p} : = \bar{V}^{p}_{1} \oplus \bar{V}^{p}_{2} \to \cO_{\PP}$.

Let $\iota : Z(\bar{\sigma}) \hra  Z(\bar{\sigma}_2)$ be the inclusion of degeneracy loci of cosections.
We note that $Z(\bar{\sigma}) = \wtil{\cM}^{rat}_{Q} $ and $Z(\bar{\sigma}_2) = \bar{\gamma}^{-1}(Z(\bar{\sigma}))$, which is a $\PP^m$-bundle $\PP|_{\wtil{\cM}^{rat}_{Q}}$ over $\wtil{\cM}^{rat}_{Q}$. 
Moreover, $\iota(Z(\bar{\sigma}))$ is the zero section of the $\PP|_{\wtil{\cM}^{rat}_{Q}}$.
The following proposition is important for localizing target space $\PP^n$ to $Q$.  
\begin{prop}\label{anallocalGysin}
\[
\iota_* \Gysin{\bar{V}^{p},\bar{\sigma}} = \Gysin{\bar{V}_2^{p},\bar{\sigma}_2}\circ \Gysin{\bar{V}_1^{p}} : A_*(\bar{V}^{p}(\bar{\sigma})) \to A_*(Z(\bar{\sigma}_2))_{\QQ} \cong A_*(\PP|_{\wtil{\cM}^{rat}_Q})_\QQ.
\]
where $\Gysin{\bar{V}_1^{p}}$ is in fact a bivariant class from $\bar{V}^{p}(\bar{\sigma})$ to $\bar{V}^{p}_2(\bar{\sigma}_2)$ for the following fiber diagram :
\[
\xymatrix{
\bar{V}^{p}(\bar{\sigma}) \ar@{^(->}[r] & |\bar{V} ^{p}| \ar[r] & |\bar{V}^{p}_1| \\
\bar{V}^{p}_2(\bar{\sigma}_2) \ar@{^(->}[r] \ar@{^(->}[u] & |\bar{V}^{p}_2| \ar@{^(->}[u] \ar[r] & \PP \ar@{^(->}[u]_{0_{\bar{V}_1^{p}}}.
}
\]
\end{prop}
\begin{proof}
Let $q : |\bar{V}^{p}| \to \PP$ be the projection, and $t_{\bar{V}^{p}}$ be the tautological section of $q^* \bar{V}^{p}$. 
Consider $w_{\bar{\sigma}} :=  q^* \bar{\sigma} \circ t_{\bar{V}^{p}} \in \Gamma(\cO_{|\bar{V}^{p}|})$. 
Recall that $w_{\bar{\sigma}}^{-1}(0) = \bar{V}^{p}(\bar{\sigma})$. 
By \cite[Theorem 2.6]{KO18localized}, we have
\[
\Gysin{\bar{V}^{p},\bar{\sigma}} = \mathrm{td}(\bar{V}^{p}|_{Z(\bar{\sigma})})\cdot \mathrm{ch}^{\bar{V}^{p}(\bar{\sigma})}_{Z(\bar{\sigma})}(\{q^*\bar{\sigma},t_{\bar{V}^{p}}\})
\]
where $\{q^*\bar{\sigma},t_{V} \}$ is a Koszul $2$-periodic complex of vector bundles on $\bar{V}^{p}(\bar{\sigma})$ induced by $q^*\bar{\sigma}$ and $t_{\bar{V}^{p}}$; see \cite{PV00topchern, KO18localized} for the notations.
We will use so called $\mathbb{A}^1$-homotopy invariance of localized Chern characters in \cite[Lemma 2.1]{PV00topchern}.

Consider the projections $q'_1 : |\bar{V}^{p}| \to |\bar{V}_{1}^{p}|$, $q'_2 : |\bar{V}^{p}| \to |\bar{V}_{2}^{p}|$ and $q_1 : |\bar{V}_{1}^{p}| \to \PP$, $q_2 : |\bar{V}_{2}^{p}| \to \PP$. 
Also consider the tautological sections $t_{\bar{V}_{1}^{p}}$ of $q_1^* \bar{V}_{1}^{p}$ and $t_{\bar{V}_{2}^{p}}$ of $q_2^* \bar{V}_{2}^{p}$. 
We have $t_{\bar{V}^{p}} = (q'_1)^{*} t_{\bar{V}_{1}^{p}} \oplus (q'_2)^{*} t_{\bar{V}_{2}^{p}}$. 
Moreover, since $\bar{\sigma} = \bar{\sigma}_1 \oplus \bar{\sigma}_2$, we have 
$\{q^*\bar{\sigma},t_{\bar{V}^{p}}\} = \{q^*\bar{\sigma}_2,(q'_2)^*t_{\bar{V}_{2}^{p}}\} \otimes \{q^*\bar{\sigma}_1,(q'_1)^*t_{\bar{V}_{1}^{p}}\}$.
Consider the $\mathbb{A}^1$-family of a 2-periodic complex of vector bundles on $\bar{V}^{p}(\bar{\sigma}) \times \mathbb{A}^1$,
$$\{q^*\bar{\sigma}_2,(q'_2)^*t_{\bar{V}_{2}^{p}}\} \otimes \{\lambda q^*\bar{\sigma}_1,(q'_1)^*t_{\bar{V}_{1}^{p}}\}$$
where $\lambda \in \mathbb{A}^1$ is an $\mathbb{A}^1$-parameter. 
By \cite[Lemma 2.1]{PV00topchern}, we have
\begin{align*} \label{Decomp:Ch}
\iota_* \mathrm{ch}^{\bar{V}^{p}(\bar{\sigma})}_{Z(\bar{\sigma})} (\{q^*\bar{\sigma},t_{\bar{V}^{p}}\} ) &=
\mathrm{ch}^{\bar{V}^{p}(\bar{\sigma})}_{Z(\bar{\sigma}_2)} (\{q^*\bar{\sigma},t_{\bar{V}^{p}}\} ) \\
&= \mathrm{ch}^{\bar{V}^{p}(\bar{\sigma})}_{Z(\bar{\sigma}_2)} (\{q^*\bar{\sigma}_2,(q'_2)^*t_{\bar{V}_{2}^{p}}\} \otimes \{q^*\bar{\sigma}_1,(q'_1)^*t_{\bar{V}_{1}^{p}}\} ) \nonumber \\
&= \mathrm{ch}^{\bar{V}^{p}(\bar{\sigma})}_{Z(\bar{\sigma}_2)} (\{q^*\bar{\sigma}_2,(q'_2)^*t_{\bar{V}_{2}^{p}}\} \otimes \{0,(q'_1)^*t_{\bar{V}_{1}^{p}}\} ). \nonumber
\end{align*}
Thus, we obtain
\begin{equation} \label{Decomp:Ch}
\iota_* \Gysin{\bar{V}^{p},\bar{\sigma}} = \mathrm{td}(\bar{V}^{p}|_{Z(\bar{\sigma}_2)})\cdot \mathrm{ch}^{\bar{V}^{p}(\bar{\sigma})}_{Z(\bar{\sigma}_2)} (\{q^*\bar{\sigma}_2,(q'_2)^*t_{\bar{V}_{2}^{p}}\} \otimes \{0,(q'_1)^*t_{\bar{V}_{1}^{p}}\} ).
\end{equation}
By \cite[Proposition 2.3(vi)]{PV00topchern},
\begin{align*}
&\mathrm{ch}^{\bar{V}^{p}(\bar{\sigma})}_{Z(\bar{\sigma}_2)} (\{q^*\bar{\sigma}_2,(q'_2)^*t_{\bar{V}_{2}^{p}}\} \otimes \{0,(q'_1)^*t_{\bar{V}_{1}^{p}}\} ) \\
&= \mathrm{ch}^{\bar{V}^{p}(\bar{\sigma})}_{q_1^{-1}(Z(\bar{\sigma}_2))} (\{q^*\bar{\sigma}_2,(q'_2)^*t_{\bar{V}_{2}^{p}}\}) \circ \mathrm{ch}^{\bar{V}^{p}(\bar{\sigma})}_{q_2^{-1}(\PP ) \cap \bar{V}^{p}(\bar{\sigma})} (\{0,(q'_1)^*t_{\bar{V}_{1}^{p}}\}).
\end{align*}
Note that $q_2^{-1}(\PP ) \cap \bar{V}^{p}(\bar{\sigma}) =\bar{V}^{p}_2(\bar{\sigma}_2)$.
Thus, the equation \eqref{Decomp:Ch} becomes
\begin{align*}
\iota_* \Gysin{\bar{V}^{p},\bar{\sigma}} =& \mathrm{td}(\bar{V}^{p}|_{Z(\bar{\sigma}_2)})\cdot \mathrm{ch}^{\bar{V}^{p}(\bar{\sigma})}_{Z(\bar{\sigma}_2)} (\{q^*\bar{\sigma}_2,(q'_2)^*t_{\bar{V}_{2}^{p}}\} \otimes \{0,(q'_1)^*t_{\bar{V}_{1}^{p}}\} ) \\
 = & \mathrm{td}(\bar{V}_1^{p}|_{Z(\bar{\sigma}_2)})\cdot \mathrm{ch}^{\bar{V}^{p}(\bar{\sigma})}_{q_1^{-1}(Z(\bar{\sigma}_2))} (\{q^*\bar{\sigma}_2,(q'_2)^*t_{\bar{V}_{2}^{p}}\}) \\
& \circ 
\mathrm{td}(\bar{V}_2^{p}|_{Z(\bar{\sigma}_2)})\cdot \mathrm{ch}^{\bar{V}^{p}(\bar{\sigma})}_{q_2^{-1}(\PP ) \cap \bar{V}^{p}(\bar{\sigma})} (\{0,(q'_1)^*t_{\bar{V}_{1}^{p}}\})\\
=&  \Gysin{\bar{V}_2^{p},\bar{\sigma}_2}\circ \Gysin{\bar{V}_1^{p}} .
\end{align*}
Here, the last equality comes from \cite[Theorem 2.6]{KO18localized}.
\end{proof}

As a corollary, we have the following proposition. %which is a modification of \cite[Corollary 6.5]{CL15}.
For an integral closed substack $Z$ of $V^{p}(\sigma) \subset   |V^{p}|$, let $\bar{Z}$ be the closure in $|\bar{V}^{p}|$. 
Here, we identify $|V^{p}|$ with $|\bar{V}^{p} |_{\PP \backslash D_\infty} |$, which is a dense open subset of $|\bar{V}^{p}|$.
Let $\bar{Z}_b:=\bar{Z}\cap |  0\oplus \bar{V}^{p}_2  |$. 
We consider a normal cone $C_{\bar{Z}_b/\bar{Z}}$, which is naturally embedded in $|\bar{V}^{p}|$ as a closed substack.
\begin{prop}\label{anallocalGysin2}
%In the above setting, 
We have 
\begin{align*}
\Gysin{V^{p}, \sigma} [Z] = \bar{\gamma}_* \Gysin{\bar{V}^{p}_{2},\bar{\sigma}_2}  \Gysin{\bar{V}^{p}_{1}} [\bar{Z}]  &= \bar{\gamma}_* \Gysin{\bar{V}^{p}_{2},\bar{\sigma}_2}\Gysin{\bar{V}^{p}_{1}} [C_{\bar{Z}_b/\bar{Z}}] \\
& \in A_*(Z(\bar{\sigma}))_\QQ = A_*(\wtil{\cM}^{rat}_Q )_\QQ.
\end{align*}
\end{prop}
\begin{proof}
By applying $\bar{\gamma}_*$ to Proposition \ref{anallocalGysin}, we have
\[
\bar{\gamma}_* \iota_* \Gysin{\bar{V}^{p},\bar{\sigma}} = \bar{\gamma}_* \Gysin{\bar{V}_2^{p},\bar{\sigma}_2}\circ \Gysin{\bar{V}_1^{p}}.
\]
Since $\iota \circ \bar{\gamma} = \id$, we have
\[
\Gysin{\bar{V}^{p},\bar{\sigma}}[\bar{Z}] = \bar{\gamma}_* \Gysin{\bar{V}_2^{p},\bar{\sigma}_2}\circ \Gysin{\bar{V}_1^{p}}[\bar{Z}]
\]
for any irreducible sublocus $Z \subset V^{p}(\sigma)$. 
It is clear that 
$\Gysin{\bar{V}^{p},\bar{\sigma}}[\bar{Z}] = \Gysin{V^{p},\sigma}[Z]$ by the compatibility with the flat pullback $i : |A_2| \hra \PP$ and the bivariant class $\Gysin{\bar{V}^{p},\bar{\sigma}}$. The second equality comes from the following equality:
\begin{align*}
\Gysin{\bar{V}^{p}_{1}} [\bar{Z}] = \Gysin{\bar{V}^{p}_{1}} [C_{\bar{Z}_b/\bar{Z}}] \text{ in } A_*(\bar{V}^{p}_2(\bar{\sigma}_2))_\QQ .
\end{align*}
Indeed, it follows from \cite[Lemma 2.1]{PV00topchern} for an $\mathbb{A}^1$-family of classes that
$$\mathrm{ch}^{M_{\bar{Z}_b / \bar{Z}}}_{\bar{Z}_b \times \mathbb{A}^1 } ( \phi^* \{ 0, (q'_1)^* t_{\bar{V}_1^{p}  }\} ) [M_{\bar{Z}_b / \bar{Z}}] \in A_*(\bar{Z}_b \times \mathbb{A}^1)_\QQ$$
where $M_{\bar{Z}_b / \bar{Z}}$ is the space of the deformation of $\bar{Z}$ to the normal cone $C_{\bar{Z}_b / \bar{Z}}$ and $\phi: M_{\bar{Z}_b / \bar{Z}} \subset |\bar{V}^{p}| \times \mathbb{A}^1 \to |\bar{V}^{p}| $ is the composition of an inclusion and a projection.
Note that $\bar{Z}_b \subset \bar{V}_2^{p}(\bar{\sigma}_2)$ since we take $Z \in Z_*(V^{p}(\sigma))$.
\end{proof}

For a decomposition by integral cycles $[C^{p,\mu}] = \sum_i [W^\mu_i] \in A_* (V^{p} (\sigma  ))$, $\mu \in \fS$, let
$R^{\mu}_1 := \sum_i [C_{(\bar{W}^{\mu}_{i})_b/\bar{W}^{\mu}_{i} }]$.
By Proposition \ref{anallocalGysin2}, we have
\begin{align} \label{FIRST:CONE}
\Gysin{V^{p}, \sigma} [C^{p,\mu}] =  \bar{\gamma}_* \Gysin{\bar{V}^{p}_{2},\bar{\sigma}_2}\Gysin{\bar{V}^{p}_{1}} [R^\mu_1].
\end{align}
By Proposition \ref{conesupp}, the cycle $\Gysin{\bar{V}^{p}_{1}} [R^\mu_1]$ lies on  $\wtil{\cM}^{p, rat}$.
Let $B^\mu_1 := \Gysin{\bar{V}^{p}_{1}} [R^\mu_1] \in A_*(\wtil{\cM}^{p, rat})_\QQ$.
Similarly, for a decomposition by integral cycles $[C^{p,\mu}_\Delta] = \sum_i [Z^\mu_i] \in A_* (V^{p} (\sigma  ))$, $\mu \in \fS$, let
$R^{\mu}_2 := \sum_i [C_{(\bar{Z}^{\mu}_{i})_b/\bar{Z}^{\mu}_{i} }]$.
By Proposition \ref{anallocalGysin2}, we have
\begin{align} \label{SECOND:CONE}
\Gysin{V^{p}, \sigma} [C^{p,\mu}_\Delta] =  \bar{\gamma}_* \Gysin{\bar{V}^{p}_{2},\bar{\sigma}_2}\Gysin{\bar{V}^{p}_{1}} [R^\mu_2].
\end{align}
By Proposition \ref{conesupp}, the cycle $\Gysin{\bar{V}^{p}_{1}} [R^\mu_2]$ lies on $|\bar{\gamma}^*F|$.
Let $B^\mu_2:=\Gysin{\bar{V}^{p}_{1}} [R^\mu_2] \in A_*(|\bar{\gamma}^*F| )_\QQ.$

Now, we compute the dimensions of $B^{\mu}_{ i}$. 
By using local defining equations, we see that $\wtil{\cM}^{\mu}$ are pure with the same dimension for all $\mu$. 
Thus, we may assume that $\mu = \{(d, [k])\}$. 
We have
\begin{align*}
 \dim \wtil{\cM}^{\mu}  = \dim\wtil{\cM}^{\{(d,[k])\}} &= \dim(\bar{M}_{0,k+1}(\PP^n,d) \times \bar{M}_{1,1})\\
 %&=(d+1)(n+1)-1+(k-2)+1 
 & = (d+1)(n+1)+k-2.\nonumber
\end{align*}
Hence, we have $\dim\wtil{\cM}^{p,\mu} = (d+1)(n+1)+k+m-2$. 
By Proposition \ref{conestr}, we deduce that $C^{\mu}$ has dimension $(d+1)(n+1)+k+m$. 
On the other hand, we have $\rank V_{1}^{p} = \rank V_1 = n+1$. 
Therefore, we have
\begin{equation} \label{dimensioncount}
B^\mu_{1} \in A_{d(n+1)+k+m}(\wtil{\cM}^{p,rat}), \ B^\mu_{2} \in A_{d(n+1)+k+m}(| \bar{\gamma}^* F  | ).
\end{equation}
By \eqref{CONE:DECOMP}, \eqref{FIRST:CONE}, \eqref{SECOND:CONE} and Proposition \ref{anallocalGysin2}, we have 
\begin{align} \label{middle:summary}
\Gysin{\bdst{E^p},\sigma}[\fC^{\mu}] = \Gysin{V^{p}, \sigma  }[C^{\mu}] = \bar{\gamma}_* \Gysin{\bar{V}^{p}_{2},\bar{\sigma}_2}(B^{\mu}_1)  + \bar{\gamma}_*\Gysin{\bar{V}^{p}_{2},\bar{\sigma}_2}(B^{\mu}_{2}).
\end{align}

\subsubsection{Contribution of $\bar{\gamma}_*\Gysin{\bar{V}^{p}_{2},\bar{\sigma}_2}(B^{\mu}_{2})$}

Now, we will show that $\bar{\gamma}_*\Gysin{\bar{V}^{p}_{2},\bar{\sigma}_2}(B^{\mu}_{2}) = 0$ for each $\mu \in \fS$.
Consider the projection $| \bar{\gamma}^* F| \to |F|$ induced by the projection $\bar{\gamma} : \PP \to \wtil{\cM}^{rat}$. 
Note that $\bar{\gamma}$ is proper. 
We have $\bar{\gamma}_*\Gysin{\bar{V}^{p}_{2},\bar{\sigma}_2}(B^{\mu}_{2}) = \Gysin{V_{2},\xi_2}(\bar{\gamma}_* B^{\mu}_{2})$. 
By \eqref{dimensioncount}, $\bar{\gamma}_* B^\mu_{2} \in A_{d(n+1)+k+m}( | F  |  )$. 
We can observe that the dimension of $|F |$ is equal to \begin{align*}  
\dim \Theta+ m = \dim \wtil{\cM}^\mu- n+m
= d(n+1) + k+ m -1.
\end{align*}
So we have $A_{d(n+1)+k+m}( | F | )=0$ and $\bar{\gamma}_* B^\mu_{2}=0$.

\subsubsection{Contribution of $\bar{\gamma}_*\Gysin{\bar{V}^{p}_{2},\bar{\sigma}_2}(B^{\mu}_{1})$}

%We compute $\bar{\gamma}_*\Gysin{\bar{V}^{\vec{p}}_{2},\bar{\sigma}_2}(B_{\mu,i})$ for $i=1,2$ in the following, which is a direct analogue of \cite[Lemma 8.1]{CL15}.
\begin{prop}\label{vanishing}
If $Q$ is 2-fold, 
\begin{align} \label{VANISHING}
\bar{\gamma}_* \Gysin{\bar{V}^{p}_{2},\bar{\sigma}_2}(B^\mu_{1}) \cap \prod\limits_{j=0}^k ev_j^*(\alpha_j)=0
\end{align}
for any cohomology classes $\alpha_j \in H^*(Q, \QQ)$. 
The same equation \eqref{VANISHING} holds true if $Q$ is 3-fold and $\mu \neq \{(d,[k])\}$.
\end{prop}
\begin{proof}
Let $\mu = \{(d_1,L_1),...,(d_{\ell},L_{\ell})\}$, $k_i := |L_i|$, $k_0 := k- \sum\limits_{i=1}^{\ell} k_i$ and $S_{\mu}$ be a subgroup of the symmetric group $S_{\ell}$, which permutes $(d_1,\dots,d_{\ell})$ but does not changes the element $\mu$. 
If $\mu$ is not of the form $\{(d,L_1)\}$, %for each $\mu \in \fS$, 
we define 
\[
\what{X}_{\mu} := \left(\bar{M}_{0,k_1+1}(\PP^n,d_1) \times_{\PP^n} \dots \times_{\PP^n} \bar{M}_{0,k_{\ell}+1}(\PP^n,d_{\ell})\right)
\]
and define $X_{\mu}:= \what{X}_{\mu}/S_{\mu}$
where each morphism $\bar{M}_{0,k_j+1}(\PP^n,d_j) \to \PP^n$ is a $(k_j+1)$-th evaluation map. 
In particular, when $\mu$ is of the form $\{(d,L_1)\}$, $X_{\mu}$ is isomorphic to $\bar{M}_{0,k_1}(\PP^n,d)$.

Next, we construct a proper morphism $r_{\mu} : \wtil{\cM}^{\mu} \to X_{\mu}$ by the following. 
For a closed point $x \in \wtil{\cM}^{\mu}$, let $\cC_x := \cC_{\wtil{\cM}^{\mu}} \times_{\wtil{\cM}^{\mu}} \{x\}$ and let $f_x : \cC_x \to \PP^n$ be the tautological morphism. 
The curve $\cC_x$ is a union of the genus one subcurve and $\ell$ genus $0$ curves $\cC_{x,1},\dots,\cC_{x,\ell}$ attached to the genus 1 subcurve. 
Let $f_{x,j} := f_x|_{\cC_{x,j}}$, where there are $k_j$ marked points on each of them. 
For each $\cC_{x,j}$ we can consider $k_j+1$ marked points, where first $k_j$-points are original ones and the last $(k_j+1)$-th point is a node which meets the genus one curve. 
Here, we define $r_{\mu}(x) := [ \ (\cC_{x,1},f_{x,1}),\dots,(\cC_{x,\ell},f_{x,\ell}) \ ] \in X_{\mu}$.

%In a similar manner as in the proof of \cite[Lemma 8.1.(2)]{CL15}, 
We construct a rank $(\sum^m_{i=1} \deg f_i)d + m$-vector bundle $W_{\mu}$ on $X_{\mu}$ such that $r_{\mu}^* W_{\mu}\dual \cong V_{2}$. 
Let $\pi_j : \cC_j \to \bar{M}_{0,k_j + 1}(\PP^n,d_j)$ be a universal curve, $p_j : \what{X}_{\mu} \to \bar{M}_{0,k_j + 1}(\PP^n,d_j)$ be the natural projection, $\phi_j : \cC_j \to \PP^n$ be the universal morphism, and $ev_j : \bar{M}_{0,k_j + 1}(\PP^n,d_j) \to \PP^n$ be the evaluation morphism at $(k_j + 1)$-th marked point. 
Since $\what{X}_{\mu}$ is defined as a fiber product via $(k_j +1)$-th evaluation maps, there is an induced single evaluation morphism $ev : \what{X}_{\mu} \to \PP^n$.
Let us define $W_{j,i}:=\pi_{j *}\phi_j^*\cO_{\PP^n}(\deg f_i)$.
We construct a vector bundle $\what{W}_{\mu,i}$ on $\what{X}_{\mu}$ which fits into the following short exact sequence
\[
\ses{\what{W}_{\mu,i}}{\oplus^{\ell}_{j=1}p_j^*W_{j,i}}{ev^*\left(\cO(\deg f_i)^{\oplus\ell}/\cO(\deg f_i)\right)}.
\]
Next we define a section $\wtil{\alpha}\dual_i : \cO \to \what{W}_{\mu,i}$. 
%Let us recall that $(f_1,\dots,f_i,\dots,f_m)$ are defining equations of $Q\subset \PP^n$.
Consider the following diagram
\[
\xymatrix{
0 \ar[r] & \what{W}_{\mu,i} \ar[r] & \oplus^{\ell}_{j=1} p_j^* W_{j,i} \ar[r]^(0.37){e} &  ev^*\left(\cO(w_i)^{\oplus\ell}/\cO(w_i)\right) \ar[r] & 0 \\
& \cO \ar@{..>}[u]_{\wtil{\alpha}\dual_i} \ar[r]_{\triangle_i} & \cO^{\oplus \ell} \ar[u]_{\oplus_j f_i} & &
}
\]
where $\oplus_j f_i$ is a map sending each component of the standard basis $e_j$ to $f_i \in p_j^*W_{j,i} = p_j^*(\pi_* \phi_j^*(\cO_{\PP^n}(\deg f_i)))$ and $\triangle_i$ is the diagonal morphism. Since $e\circ \oplus_j f_i \circ \triangle_i = 0$, there exists the induced morphism $\wtil{\alpha}\dual_i : \cO \to \what{W}_{\mu,i}$.
Thus we have $\wtil{\alpha}^\vee:= \oplus_i \wtil{\alpha_i}^\vee : \cO \to \what{W}_\mu := \oplus_i \what{W}_{\mu,i}$.
We see that $\what{W}_\mu$ and $\wtil{\alpha}^\vee$ are $S_\mu$-equivariant.
Let $W_\mu:=\what{W}_\mu / S_\mu$ be a vector bundle on $X_\mu$ and $\alpha:=\wtil{\alpha} /S_\mu : W_\mu^\vee \to \cO$ be a cosection.
We can check that $r_\mu^* W_\mu^\vee \cong V_2$ and $r_\mu^* \alpha = \xi_2$.
We note that $Z(\alpha) = \left(\bar{M}_{0,k_1+1}(Q,d_1) \times_{\PP^n} \dots \times_{\PP^n} \bar{M}_{0,k_{\ell}+1}(Q,d_{\ell})\right) / S_\mu$.

Let $\nu_{\mu}$ be a proper morphism defined by the composition $\PP \stackrel{\bar{\gamma}}{\lra} \wtil{\cM}^{\mu} \stackrel{r_{\mu}}{\lra} X_{\mu}$. 
There is an induced morphism of total spaces of vector bundles %given by the composition 
$$\nu'_{\mu}: |\bar{V}^{p}_{2}| = |\bar{\gamma}^*V_{2}| \to |V_{2}|=|r_{\mu}^* W_{\mu}\dual| \to |W_{\mu}\dual|.$$ 
Finally, we have
\[ (\nu_{\mu})_*  \Gysin{\bar{V}_{2}^{p},\bar{\sigma}_2}(B^{\mu}_1) = \Gysin{W_{\mu}\dual, \alpha}((\nu'_{\mu})_*(B^{\mu}_1)).
\]
%where $i : (\bar{\sigma}_2\dual)^{-1}(0) \to \PP_\mu$ is an inclusion.
We can check that $(\nu'_{\mu})_*(B^{\mu}_{1}) \in A_{(n+1)d+k+m}(X_{\mu})$ and that

\[
\dim X_{\mu} =
\left\{ \begin{array}{ll}
(n+1)d-2\ell+n+k-k_0 & \textrm{if $\mu$ is not of the form $\{(d,L_1)\}$} \\
(n+1)d+(n-3)+k-k_0 & \textrm{if $\mu$ is of the form $\{(d,L_1)\}$}.
\end{array} \right.
\]

\medskip

\noindent Recall that $\dim Q=n-m$.
If $Q$ is $2$-fold, then $\dim X_{\mu}$ is strictly smaller than $(n+1)d+k+m$. 
Hence, $(\nu_{\mu})_* \Gysin{\bar{V}_{2}^{p},\bar{\sigma}_2}(B^{\mu}_{1}) = 0$.
If $Q$ is $3$-fold, $(\nu_{\mu})_*  \Gysin{\bar{V}_{2}^{p},\bar{\sigma}_2}(B^{\mu}_{1})= (r_{\mu})_* \bar{\gamma}_*  \Gysin{\bar{V}_{2}^{p},\bar{\sigma}_2}(B^{\mu}_{1})=0$ except for the case when $\mu=\{(d,[k])\}$. 
Let us abbreviate $\prod\limits_{j=0}^k ev_j^*(\alpha_j)$ by $ev^*(\alpha)$. 
Then we have
\begin{align*}
 \bar{\gamma}_*  \Gysin{\bar{V}^{p}_{2},\bar{\sigma}_2}(B^{\mu}_{1})\cap ev^*(\alpha) &= (r_{\mu})_* (\bar{\gamma}_*  \Gysin{\bar{V}^{p}_{2},\bar{\sigma}_2}(B^{\mu}_{1})\cap ev^*(\alpha)  )
\\
&= (r_{\mu})_* \bar{\gamma}_* (\Gysin{\bar{V}^{p}_{2},\bar{\sigma}_2}(B^{\mu}_{1})\cap ev^*(\alpha)  )  \\
& =(\nu_\mu)_*    (\Gysin{\bar{V}^{p}_{2},\bar{\sigma}_2}(B^{\mu}_{1})\cap ev^*(\alpha)  )\\
& =\Gysin{W_{\mu}\dual,\alpha}((\nu_{\mu})_*(B^{\mu}_{1})) \cap  ev^*(\alpha) = 0
\end{align*}
unless $Q$ is 3-fold and $\mu=\{(d,[k])\}$.
\end{proof}

%In this section we compute nonvanishing terms in the contribution from rational components. 
%For the case $\dim Q=3$ and $\mu = \{(d,[k])\}$, by the proof of Proposition \ref{vanishing}, we have the following result. 
%From now on, we assume .
\begin{prop} \label{Cycle:Equi}
When $Q$ is a $3$-fold and $\mu = \{(d,[k])\}$, we have
\begin{align*} 
(r_\mu)_* \bar{\gamma}_* \Gysin{\bar{V}^{p}_{2}, \bar{\sigma}_2 }(B^{\mu}_{1}) 
= (-1)^{d\sum_i \deg f_i+m} c [\bar{M}_{0,k}(Q,d)]^{vir}
\end{align*}
in $A_*(\bar{M}_{0,k}(Q,d))_\QQ$
for some constant $c \in \QQ$.
\end{prop}

We note that the sign comes from the rank of $W_\mu^\vee$.
%difference between $e(W_\mu)$ and $e(W^\vee_\mu)$ where $e(-)$ stands for the Euler class.

\begin{proof}
Since $\bar{M}_{0,k}(\PP^n,d)$ is irreducible by \cite{KP01conn}, we have 
$$(\nu'_\mu)_*B^{\mu}_{1} = c[\bar{M}_{0,k}(\PP^n,d)]$$
by dimension reason.
We apply $0^!_{W_\mu^\vee, \alpha}$ on both side.
The left hand side is 
$$0^!_{W_\mu^\vee, \alpha} ((\nu'_\mu)_*B^{\mu}_{1}) 
=(\nu_\mu)_* (0^!_{\bar{V}_2^{p}, \bar{\sigma}_2} (B^{\mu}_{1}))
=(r_\mu)_* (\bar{\gamma})_* (0^!_{\bar{V}_2^{p}, \bar{\sigma}_2} (B^{\mu}_{1}))$$
by the bivariant property of localized Chern characters.
We claim that the right hand side is 
\begin{align*}
c\cdot  0^!_{W_\mu^\vee, \alpha}  [\bar{M}_{0,k}(\PP^n,d)] & = (-1)^{d\sum_i \deg f_i+m} c\cdot  (\alpha^\vee)^!  [\bar{M}_{0,k}(\PP^n,d)]\\
&= (-1)^{d\sum_i \deg f_i+m} c \cdot [\bar{M}_{0,k}(Q,d)]^{vir}.
\end{align*}
\end{proof}

\begin{coro} \label{nonvanishing}
When $Q$ is a $3$-fold and $\mu = \{(d,[k])\}$, we have
\begin{align*} 
\bar{\gamma}_* \Gysin{\bar{V}^{p}_{2}, \bar{\sigma}_2 }(B^{\mu}_{1})\cap ev^*(\alpha)
= (-1)^{d\sum_i \deg f_i+m} c [\bar{M}_{0,k}(Q,d)]\cap ev^*(\alpha)
\end{align*}
where $c \in \QQ$ is the constant introduced in Proposition \ref{Cycle:Equi}.
\end{coro}
%\begin{proof}
%By the proof of Proposition \ref{vanishing}, we have
%\[
%\bar{\gamma}_* \Gysin{\bar{V}^{\vec{p}}_{2}, \bar{\sigma}_2 }(B_{\mu,1})\cap \vec{ev}^*(\vec{h}) = \Gysin{W_{\mu},\alpha}((\nu_{\mu})_*(B_{\mu,1})) \cap r_\mu^* \vec{ev}^*(\vec{h}).
%\]
%By the fact that the stable map space $\bar{M}_{0,{\color{red}k}}(\PP^n,d)$ is irreducible and the hyperplane property of genus $0$ Gromov-Witten invariants in \cite{Kon95enu} we have
%\begin{align*} 
%&\Gysin{W_{\mu},\alpha}((\nu_{\mu})_*(B_{\mu,1})) \cap r_\mu^* \vec{ev}^*(\vec{h}) \\
%&= i_*(\Gysin{W_{\mu},\alpha}((\nu_{\mu})_*(B_{\mu,1})) \cap r_\mu^* \vec{ev}^*(\vec{h}) ) \\
%&= (-1)^{(\sum_i w_i)d+m} c [\bar{M}_{0,{\color{red}k}}(Q,d)]\circ \vec{ev}^*(\vec{h})
%\end{align*}
%for some constant $c \in \QQ$ where $i: \bar{M}_{0,{\color{red}k}}(Q,d) \to \bar{M}_{0,{\color{red}k}}(\PP^n, d)$ is an inclusion. 
%\end{proof}

Next, we compute the constant $c$. We follow the idea in \cite[Section 9]{CL15}. 
We fix a $k$ marked points $x_1,\dots,x_k \in \PP^1$ and a degree $d$ regular embedding $g : \PP^1 \to \PP^n$. 
Let $R =\bar{M}_{1,1} \times \PP^1$  where $\bar{M}_{1,1}$ is the moduli space of genus one stable curves with one marked point. 
There is a closed embedding $\psi_g : R \to \bar{M}_{1,k}(\PP^n,d)$ defined by the following: 
For $((E,a),b ) \in R$, let $C$ be an elliptic curve with $k$ marked points obtained by gluing $E$ and $(\PP^1,x_1,\dots,x_k)$ which identifies $a$ and $b$, and take stabilization if $b$ is equal to one of $x_1,\dots,x_k$. 
We write $C$ as $C=E\cup T$, where %$E$ is an irreducible elliptic curve and 
$T$ is a genus $0$ nodal curve. 
If $b$ is not equal to one of $x_1,\dots,x_k$, then $T = \PP^1$ and if not, then $T = \PP^1 \cup \PP^1$.
For latter case, we denote $(\PP^1)_1$ the bridge $\PP^1$ component in $T$ and $(\PP^1)_2$ the other component.
We define a morphism $g_C : C \to \PP^1$ such that $g_C|_E$ and $g_C|_{(\PP^1)_1}$ (if exists) are constant and $g|_{\PP^1} \cong g$ (or $g|_{(\PP^1)_2} \cong g$), where $\PP^1$ is a component in $C$ containing $k$ ordered points, which are equal to $(x_1,\dots,x_k)$ or one of $x_i$ replaced to the node point of the $\PP^1$. 
Here, $g|_{\PP^1} \cong g$ means isomorphism up to automorphism of $\PP^1$ preserving $k$ ordered points. 
We define $\psi_g((E,a),b) = (C,g_C)$.

We can consider $R$ as a closed substack of $\bar{M}_{1,k}(\PP^n,d)$. 
Let $\pi: \cC_R \to R$ be a universal curve and let $ev : \cC_R \to \PP^n$ be a universal morphism. 
Let $L_{R,i}:=(\pi)_*(ev^*\cO(-\deg f_i) \otimes \omega_{\pi})$. %where $\pi_R : \cC_R \to R$ is the projection. 
Then $L_{R,i}$ are line bundles. 
We define a rank $m$ vector bundle $A_R := \oplus^m_{i=1} L_{R,i}$. 
Let $\PP_R :=  \PP(A_{R} \oplus \cO_R)$. 
Let $D_R := \PP(A_R \oplus 0)$ be a divisor at infinity. 
Let $\bar{\gamma}_R : \PP_R \to R$ be the projection.

Consider two projections $p_1 : R \to \bar{M}_{1,1}$ and $p_2 : R \to \PP^1$. Let $\cH = (\pi_{\bar{M}_{1,1}})_*(\omega_{\cC_{\bar{M}_{1,1}}/\bar{M}_{1,1}})$ be the Hodge bundle of $\bar{M}_{1,1}$. 
Following the idea in \cite[Section 9]{CL15}, we can check that 
\begin{align}\label{coeff}
c = \deg \left( e(\bar{\gamma}^*(p_1^* \cH^\vee \otimes p_2^*N_{g(\PP^1)/\PP^n} )(-D_R))\right)
\end{align}
where $e(-)$ stands for the Euler number.

Let $\alpha = 24 \cdot c_1(\cH) \in A^1(\bar{M}_{1,1})$ and $\beta = c_1(\cO_{\PP^1}(1)) \in A^1(\PP^1)$. The total Chern class of $N_{g(\PP^1)/\PP^n}$ is equal to $1 + ((n+1)d-2)\beta$. 
We set $\bar{\alpha} = \bar{\gamma}^*p_1^*\alpha$ and $\bar{\beta} = \bar{\gamma}^*p_2^*\beta$. 
Let $F \in A^2(\PP_R)$ be the cohomology class which is the Poincar\'e dual of the fiber of the projective bundle. 
We have $F = \bar{\alpha} \cdot \bar{\beta}$, $\bar{\alpha}^2 = \bar{\beta}^2 = 0$. 
Now, we obtain 
\begin{align*}
& c(\bar{\gamma}^*(p_1^*\cH\dual \otimes p_2^*N_{g(\PP^1)/\PP^n} )) \\
 & = 1 + ((n+1)d-2)\bar{\beta} - \frac{(n-1)}{24}\bar{\alpha} - \frac{n-2}{24}((n+1)d-2)F
\end{align*}
where $c$ stands for the total Chern class. 
Therefore we have:
\begin{align*}
 e (\bar{\gamma}^* & ( p_1^*\cH\dual \otimes p_2^*N_{g(\PP^1)/\PP^n}  )(-D_R)) \\
  = & [-D_R]^{n-1} + \l(((n+1)d-2)\bar{\beta}-\frac{n-1}{24}\bar{\alpha}\r)[-D_R]^{n-2} \\
  &- \frac{n-2}{24}((n+1)d-2)F\cdot[-D_R]^{n-3}.
\end{align*}
By the definition of Segre class, we have :
\begin{align*}
 \deg \left( e \right. & \left. (\bar{\gamma}^*  (p_1^*\cH\dual \otimes p_2^*N_{g(\PP^1)/\PP^n}  )(-D_R)) \right) \\
 = & \deg  \l( [-D_R]^{n-1}  + \l(((n+1)d-2)\bar{\beta}-\frac{n-1}{24}\bar{\alpha}\r)[-D_R]^{n-2} \right. \\
&-  \left. \frac{n-2}{24}((n+1)d-2)F\cdot[-D_R]^{n-3} \r)  \\
 = &  \deg \  \bar{\gamma}_* \l([-D_R]^{n-1} + \l(((n+1)d-2)\bar{\beta}-\frac{n-1}{24}\bar{\alpha}\r)[-D_R]^{n-2}  \right. \\
 &  \left. - \frac{n-2}{24}((n+1)d-2)F\cdot[-D_R]^{n-3} \r) \\
 = & (-1)^{n-1}s_2(A_R) + (-1)^{n-2}\l(((n+1)d-2)p_2^*\beta - \frac{n-1}{24}p_1^*\alpha \r)\cdot s_1(A_R)\\
& - (-1)^{n-3}\frac{n-2}{24}((n+1)d-2) \cdot s_0(A_R). 
\end{align*}
Here, $s_i(-)$ stands for the $i$-th Segre class.
We recall that $A_R=\oplus^m_{i=1} L_{R,i}$ and $L_{R,i}\cong p_1^*\cH \otimes p_2^*g^*\cO(-\deg f_i)$. 
Thus we have the total Segre class 
$$s(L_{R,i}) = 1 - \frac{1}{24}p_1^*\alpha + \deg f_i \cdot d \cdot p_2^*\beta - \frac{\deg f_i \cdot d}{12}[pt].$$ 
By the Whitney sum formula for Segre classes, we have:
\[ s(N_R) = \prod\limits^m_{i=1}(1 - \frac{1}{24}p_1^*\alpha + \deg f_i \cdot d \cdot p_2^*\beta - \frac{\deg f_i \cdot d}{12}[pt]).
\]
Therefore we obtain:
\begin{align*}
s_2(N_R) & = -\frac{m+1}{24}\l(\sum_i \deg f_i \r)\cdot d[pt], \\   
s_1(N_R) & =-\frac{m}{24}p_1^*\alpha + \l(\sum_i \deg f_i \r)d \cdot p_2^*\beta.
\end{align*}
Recall that we are considering the case $n-m=3$. Thus, we have 
\begin{align*}
\deg \left( e \right. & \left. (\bar{\gamma}^*  (p_1^*\cH\dual \otimes p_2^*N_{g(\PP^1)/\PP^n} )  (-D_R)) \right) \\
%& (-1)^{n-1}s_2(N_R) + (-1)^{n-2}\l(((n+1)d-2)p_2^*\beta - \frac{n-1}{24}p_1^*\alpha \r)\cdot s_1(N_R) \\
%& - (-1)^{n-3}\frac{n-2}{24}((n+1)d-2) \cdot s_0(N_R) \\
%& = (-1)^n\frac{m+1}{24}\l(\sum_i w_i \r) \cdot d + (-1)^{n-1}\frac{1}{24}\l(m((n+1)d-2) + (n-1)\l(\sum_i w_i \r)\cdot d \r) \\
%& + (-1)^{n-2}\frac{n-2}{24}((n+1)d-2) \\
& =  (-1)^n \l( \frac{n-m-2}{24}((n+1)d-2) - \l( \sum_i \deg f_i \r)d\cdot \frac{n-m-2}{24}  \r) \\
& = (-1)^{n+1}\l(\frac{1}{12} - \frac{(n+1 - \sum_i \deg f_i)\cdot d}{24} \r) \\
& = (-1)^{n+1}\l( \frac{2 - c_1(T_Q)\cdot d[line]}{24} \r)
\end{align*}
where $[line]$ is the homology class of a projective line in $\PP^n$.
By \eqref{coeff}, we have
\begin{align} \label{COEF}
c = (-1)^{n+1}\l( \frac{2 - c_1(T_Q)\cdot d[line]}{24} \r).
\end{align}

The equation \eqref{A:rat} is obtained by \eqref{DEF:A}, \eqref{middle:summary}, Section \S 5.0.1, Proposition \ref{vanishing}, Corollary \ref{nonvanishing} and \eqref{COEF} up to sign $(-1)^{n+1+m} = (-1)^{4+2m}=1$.

\bibliographystyle{plain}

\end{document}